\documentclass[12pt]{article}
\usepackage{amsmath,latexsym,amssymb,amsfonts,amsbsy, amsthm}
\usepackage[colorlinks,linkcolor=red,anchorcolor=green,citecolor=blue]{hyperref}
\usepackage[usenames]{color}
\usepackage{xspace,colortbl,float}
\usepackage{array,multirow}
\usepackage{booktabs}
\allowdisplaybreaks
\newcommand{\beq}{\begin{equation}}
\newcommand{\eeq}{\end{equation}}
\newcommand{\ben}{\begin{eqnarray}}
\newcommand{\een}{\end{eqnarray}}
\newcommand{\beno}{\begin{eqnarray*}}
\newcommand{\eeno}{\end{eqnarray*}}

\newtheorem{thm}{Theorem}[section]

\newtheorem{lem}[thm]{Lemma}
\newtheorem{prop}[thm]{Proposition}
\newtheorem{coro}[thm]{Corollary}
\newtheorem{rmk}[thm]{Remark}

\newcommand{\dif}{\mathrm{d}}

\newcommand{\lam}{\lambda}
\newcommand{\lag}{\langle}
\newcommand{\rag}{\rangle}

\begin{document}
\topmargin -7mm \oddsidemargin -1mm

\title{\Large \bf Dynamics of  nonlinear hyperbolic equations of Kirchhoff type
\author{Jianyi Chen $^{1}$,\ Yimin Sun $^{2}$,\ Zonghu Xiu $^1$,\  Zhitao Zhang $^{3,\; 4,}$
\thanks{Corresponding author.
\newline\indent The reaserch was supported by National Natural Science Foundation of China (11701310, 11771428,11926335,12031015), and the Research Foundation for Advanced Talents of Qingdao Agricultural University (6631114328,  6631115047).
\newline\indent \textit{E-mail addresses}: chenjy@amss.ac.cn (J.Y. Chen); ymsun@nwu.edu.cn (Y.M. Sun); qingda@163.com (Z.H. Xiu); zzt@math.ac.cn (Z.T. Zhang).
}\\
{\small $^{1}$ Science and Information College, Qingdao Agricultural University,}\\
 {\small Qingdao $266109$, P. R. China. }\\
 {\small $^{2}$ School of Mathematics, Northwest University, Xi'an $710127$, P. R. China.}\\
{\small $^{3}$ HLM, Academy of Mathematics and Systems Science, Chinese Academy}\\
 {\small  of Sciences, Beijing $100190$,  P. R. China.}\\
{\small   $^4$ School of Mathematical Sciences, University of Chinese  Academy }\\
{\small of Sciences, Beijing $100049$,  P. R. China }
}
}
\date{}

\renewcommand{\theequation}{\thesection.\arabic{equation}}
\setcounter{equation}{0}

\maketitle

\begin{abstract}
 In this paper, we study the initial boundary value problem of the important hyperbolic Kirchhoff equation
   $$u_{tt}-\left(a \int_\Omega |\nabla u|^2 \dif x +b\right)\Delta u = \lambda u+ |u|^{p-1}u
   ,$$
    where $a$, $b>0$, $p>1$, $\lambda \in \mathbb{R}$ and the initial energy is arbitrarily large. We prove several new theorems on the dynamics such as the boundedness or finite time blow-up of solution under the different range of $a$, $b$, $\lambda$ and  the initial data for the following cases: (i) $1<p<3$, (ii) $p=3$ and $a>1/\Lambda$, (iii) $p=3$, $a \leq 1/\Lambda$ and $\lam <b\lam_1$, (iv) $p=3$, $a < 1/\Lambda$ and $\lam >b\lam_1$,  (v) $p>3$ and $\lam\leq b\lam_1$, (vi) $p>3$ and $\lam> b\lam_1$, where $\lam_1 = \inf\left\{\|\nabla u\|^2_2 :~ u\in H^1_0(\Omega)\ {\rm and}\ \|u\|_2 =1\right\}$, and $\Lambda = \inf\left\{\|\nabla u\|^4_2 :~ u\in H^1_0(\Omega)\ {\rm and}\ \|u\|_4 =1\right\}$. Moreover, we prove the invariance of some stable and unstable sets of the solution for suitable $a$, $b$ and $\lam$, and give the sufficient conditions of initial data to generate a vacuum region of the solution. Due to the nonlocal effect caused by the nonlocal integro-differential term, we show many interesting differences between the blow-up phenomenon of the problem for $a>0$ and $a=0$.
   \\[1.2ex]
  \emph{AMS Subject Classification (2010)}: 35L20; 35B30; 35B44; 58E30.
  \\[1.1ex]
 {\emph{Keywords}}: \rm{Hyperbolic  equation; Kirchhoff equation; finite time blow-up; initial  boundary value problem; Nehari manifold}
\end{abstract}

\renewcommand{\theequation}{\thesection.\arabic{equation}}
\setcounter{equation}{0}

\section{Introduction}

\label{} \ \ \ \ \  Let $\Omega$ be an open bounded domain of $\mathbb{R}^{3}$ with smooth boundary $\partial \Omega$. We consider the following initial boundary value problem of nonlinear hyperbolic Kirchhoff equation:
\begin{eqnarray}\label{zh809-1}
  \left\{\begin{array}{l}
    u_{tt}-\left(a \int_\Omega |\nabla u|^2 \dif x +b\right)\Delta u = \lambda u+ |u|^{p-1}u,\;\  t>0,\ x\in \Omega,
    \\[1.2ex]
    u(0,x)=u_0 (x),\ u_t (0,x)=u_1 (x),\;\ \ \ x\in \Omega,
     \\[1.2ex]
       u(t,x)=0, \;\ \ \;\ \ \ \quad t\geq 0,\ x\in \partial \Omega,
      \end{array}\right.
\end{eqnarray}
 where $\nabla u$ and $\Delta u$ denote the gradient and the Laplacian of $u$ with
respect to the space variables $x=(x_1, x_2, x_3)$. Throughout this paper, we assume that $\lambda \in \mathbb{R}$, $a$, $b>0$, $p>1$ and
 \begin{eqnarray}\label{1.2}
  u_0 \in H^2(\Omega)\cap H_{0}^1 (\Omega),\ \ \ \ \ u_1 \in H_{0}^1 (\Omega).
 \end{eqnarray}

\subsection{Background}
\label{} \ \ \ \ \ For $a=0$, $b=1$, $\lambda =0$, equation \eqref{zh809-1} reduces to the nonlinear wave equation
  \begin{eqnarray}\label{1.3}
    u_{tt}- \Delta u= |u|^{p-1}u,\ \ \ x\in \Omega,
  \end{eqnarray}
which has received considerable attention since the celebrated work of Sattinger \cite{S}. Many mathematicians have devoted their effort to improve Sattinger's result; see the papers of Cazenave \cite{C}, Gazzola and Squassina \cite{Gaz}, Liu \cite{Liu}, Liu and Zhao \cite{LZ} which solve the initial-boundary value problem (1.3), and the papers of Br\'{e}zis \cite{Bre-Cor-Nir}, Chen and Zhang \cite{Ch-Zh1}-\cite{Ch-Zh3}, Ding et al. \cite{DL}, Rabinowitz \cite{Ra}, Schechter \cite{Sch1} which concern the periodic solutions for (1.3).

    When $a>0$, the problem \eqref{zh809-1} is nonlocal due to the presence of the  integro-differential term $\big(\int_\Omega |\nabla u|^2 \dif x\big) \Delta u$, which arises in many interesting models in physics, biology as well as other areas. It is well known that \eqref{zh809-1} is a natural generalization of the model
\begin{eqnarray*}
   u_{tt}-\left(a \int_0^{L} |\nabla u|^2 \dif x +b\right) u_{xx} = 0,\;\  t>0,\ x\in [0, L]
\end{eqnarray*}
 which was first proposed by Kirchhoff \cite{K} to describe the transversal vibrations of a stretched string in dimension one, where the subsequent change in length of the string was taken into consideration. Here, $u = u(t,x)$ denotes the transverse displacement of the point $x$ at the instant $t$, $L$ is the length of the string, and the parameter $b$ denotes the initial tension while $a$ is related to the intrinsic properties of the string (such as Young's modulus, the string cross-sectional area and some of  other physical quantities). In this model, Kirchhoff used the integral term $\int_{0}^L |\nabla u|^2 \dif x$ to present the average of the change in tension along the vibrating string taking account the change of the string's length. Moreover, such models can be used in tension modulations for the sound synthesis and  the control practice of mechanical systems (see \cite{WXS} for example).

    During the past decades, the Kirchhoff type problems
     \begin{eqnarray}\label{zh809-2}
      u_{tt} - m\left(\int_\Omega |\nabla u|^2 \dif x\right) \Delta u = f,
    \end{eqnarray}
    have been studied under many aspects, where $m: \mathbb{R}^+ \rightarrow \mathbb{R}^+$ is some continuous function. We refer the reader to the papers \cite{AS, B, GG1, GG2, GG3,H,MR, MR2, MR3, MS} and \cite{AP, L} for solving \eqref{zh809-2} with $f=0$ (unforced case) and $f=f(t,x)$ (linear forced case) respectively,
under various conditions on the functions $m$, $f$, the regularity and size of the initial data $u_0$, $u_1$. A functional framework to solve \eqref{zh809-2} is established by Lions in \cite{L}. It  is a challenging problem
 to study \eqref{zh809-2} with nonlinear forced term, as the interaction between nonlinearity and nonlocal effect make it more difficult and particularly interesting.\par
  When the initial data are small and analytically, D'Ancona and Spagnolo \cite{AS2}, Ghisi \cite{G1} proved the global solvability of the problem
  $$u_{tt}-\left(\int_\Omega |\nabla u|^2 \dif x +1\right)\Delta u = f(u, u_t, \nabla u)
  $$
 for $f \in C^\infty$ satisfying some growth assumptions.\par
   For the initial data without the sufficiently small and smooth assumptions, Ikehata \cite{I}, Ikehata and Okazawa \cite{IO} constructed the local solutions for \eqref{zh809-2} with $f=f(u)$ possessing polynomial growth and $m\in C^1$ satisfying $m(s)\geq m_0 >0$ for some constant $m_0$, and studied the blow-up phenomenon for \eqref{zh809-1} in particular case $\lambda =0$, $p=3$ and the initial data $(u_0, u_1)$ having the energy which is smaller than the mountain pass level
 \begin{eqnarray*}
   d = \inf\limits_{u \in H^1_0 (\Omega) \setminus \{0\}}\sup\limits_{\tau>0} J(\tau u),
 \end{eqnarray*}
 where $J(u)$, $E(0)$ are defined by \eqref{zh809-4} and \eqref{1.11} below with $p=3$, $\lam =0$. Their results depend heavily upon the behavior of $J(u)$ and the characterization of mountain pass level $d$ for \eqref{zh809-1} with $p=3$ and $\lam = 0$. Later, many progresses have been made in the well-posedness of the problem \eqref{zh809-2} with low initial energy and some damping terms; see \cite{CL, O, PPXZ, WT}.

   We would like to mention that, the well posedness of the evolution problem \eqref{zh809-1} is closely linked to the stationary state of the Kirchhoff type problem.
  Motivated by the works of \cite{CKW, ZhPe}, we find that the nature of the stationary solutions of the problem \eqref{zh809-1} is determined by the sign of number $p-3$ and the value of $\lam$. The cases $p<3$, $p=3$, $p>3$ are called {\bf 4-sublinear}, {\bf asymptotically 4-linear} and {\bf 4-superlinear} respectively (see \cite{ZhPe}). We collect some known facts on the solution $u=u(x)$ for the stationary problem corresponding to \eqref{zh809-1}
  \begin{eqnarray}\label{zh809-3}
  \left\{\begin{array}{l}
    -\left(a \int_\Omega |\nabla u|^2 \dif x +b\right)\Delta u = \lambda u+ |u|^{p-1}u\;\  {\rm in}\ \Omega,
    \\[1.2ex]
    u=0 \;\ \ \;\ \ \ \quad {\rm on}\ \partial \Omega.
      \end{array}\right.
\end{eqnarray}

  $\bullet$ Let $\Omega$ be a smooth open bounded domain of $\mathbb{R}^{3}$, Zhang and Perera \cite{ZhPe} proved that there are constants $A^* > A_* >0$ such that \eqref{zh809-3} possessing a positive solution, a negative solution and a sign changing solution in the following cases:  (i) $1<p<3$ and $\lam >b\lam_2$, (ii) $p=3$, $\lam >b\lam_2$ and $a>A^*$, (iii) $p=3$, $\lam <b\lam_1$ and $0<a<A_*$, (iv) $3<p<5$ and $\lam <b\lam_1$, where $0<\lam_1 <\lam_2 \leq \cdots$ are the Dirichlet eigenvalues of $-\Delta$ on $\Omega$.

  $\bullet$ In the case $\Omega$ is a 3 dimensional open ball, Huang et al. \cite{HLW} showed that the condition $\lam >b\lam_2$ required in \cite{ZhPe} is not necessary for constructing a positive solution of \eqref{zh809-3} with $1<p<3$. Naimen \cite{N} proved the existence, nonexistence and multiplicity of positive solutions for \eqref{zh809-3} with $p=5$ under some suitable assumptions on $a$, $b$ and $\lam$.\\
  For more results we refer the reader to papers \cite{ACM, CKW, DPS, FIJ,LLS, N, PeZh,SZ1, SZ2, TC, ZhPe} concerning with various types of nonlinearities and domains. \par\medskip

   In the present work, we are interested in studying the evolution problem \eqref{zh809-1} with arbitrary initial energy and nonlinear forced term $\lambda u + |u|^{p-1}u$ for any $p \in (1, \infty)$ and $\lambda\in \mathbb{R}$, where $\lam$ is regarded as a parameter of the linear perturbation term. We study the existence and classification of the initial data of \eqref{zh809-1} which generate the bounded solutions or finite time blow-up solutions for the following cases:\par
   (i) $1<p<3$,\par
   (ii) $p=3$ and $a>1/\Lambda$,\par
   (iii) $p=3$, $a \leq 1/\Lambda$ and $\lam <b\lam_1$,\par
   (iv) $p=3$, $a < 1/\Lambda$ and $\lam >b\lam_1$,\par
   (v) $p>3$ and $\lam\leq b\lam_1$,\par
   (vi) $p>3$ and $\lam> b\lam_1$,\\
   where the numbers $\lam_1$, $\Lambda$ are given in Section 1.2.  We investigate how the mountain pass level of \eqref{zh809-1} relies on $p$, $\lam$ and find some interesting phenomena caused by the interaction between nonlinearity and nonlocal effect of the problem \eqref{zh809-1}.

    The main contributions of our work are as follows.\par


     In the case $1<p<3$, for any $a$, $b>0$ and $\lambda \in \mathbb{R}$ we prove the boundedness of the solution for \eqref{zh809-1}
with initial dada $(u_0, u_1)$ satisfying \eqref{1.2}. Moreover, under more restrictive assumptions on $a$, $\lam$ and the initial data, we obtain a vacuum region such that there is no weak solution of \eqref{zh809-1} in this region (see Theorem 2.2 and Remark 2.3).

  For $p=3$, we can see the nonlocal coefficient $a$ may affect the behavior of the solution for \eqref{zh809-1}. When $a$, $b$, $\lambda$ lie within certain respective intervals, we
give the conditions of initial data such that the corresponding solution to \eqref{zh809-1} is bounded or blow-up in finite time. We also construct the stable and unstable sets of the solution via a careful analysis of the functional $J(u)$ and Nehari manifold $\textit{\textbf{N}}_3$ defined by \eqref{zh809-4} and \eqref{1.15} with $p=3$ respectively; see Case A and B of Theorem 2.1 (\emph{invariance of the stable and unstable sets}), Theorem 2.4 (\emph{boundedness of solution}), Theorem 2.6 and Theorem 2.8 (\emph{blow-up and a vacuum region of solution}). A striking difference is found among the following three cases: $a\geq 1/\Lambda$, $0<a<1/\Lambda$ with $\lam<b\lam_1$, and $0<a<1/\Lambda$ with $\lam > b\lam_1$.

   In the case $p > 3$, we construct some specific sets to classify the initial data to generate the bounded or blow-up solutions for \eqref{zh809-1} and prove the invariance of these sets under the flow of \eqref{zh809-1} for $3<p<5$ and $\lam \leq b\lam_1$, see Case C of Theorem 2.1. Then we study the behavior and vacuum region of solution to \eqref{zh809-1} for some suitable $b$, $\lam$ and the initial data $(u_0, u_1)$, see Theorem 2.9 and Theorem 2.10. In contrast to the case $p=3$, we show that the nonlocal coefficient $a>0$ may not affect the blow-up behavior of the solution for \eqref{zh809-1} when $p>3$.\par

     Our results reveal many striking differences between the behavior of solutions of the nonlocal problem \eqref{zh809-1} and the local problem \eqref{1.3}, which are described for details in Section 2.\par

  This paper is organized as follows. We introduce some notations in the next subsection, and state our main results in Section 2. In Sections 3 and 4, we study the geometry structure of Nehari manifold for $p=3$ and $3<p<5$ respectively, which allow the variational treatment of problem \eqref{zh809-1}. In Section 5, we provide the proofs of the main results. Finally, we prove compatibility of the assumptions for main theorems  respectively in Section 6.

\subsection{Notations and functional settings}
\label{} \ \ \ \ \
In order to proceed with the statement of our results, we require the following notations.

  We use $\|\cdot\|_q$ as the $L^q (\Omega)$ norm for $1\leq q\leq \infty$, and $\|\nabla \cdot\|_2$ as the Dirichlet norm in $H_0^1 (\Omega)$.

  Let $\lambda_1 >0$ be the principal eigenvalue of $-\Delta$ on $\Omega$ under the homogeneous Dirichlet boundary value conditions, with corresponding positive principal eigenfunction $\psi_1 (x)$. \par
  For all $u\in H_0^1(\Omega)\backslash \{0\}$, we have
   \begin{eqnarray}\label{1.6}
   b_0 \|\nabla u\|_2^2 \;\leq\; b\|\nabla u\|_2^2 - \lambda \|u\|_2^2 \;\leq\; c_1 \|\nabla u\|_2^2,
   \end{eqnarray}
 where
 \begin{eqnarray}\label{1.7}
 b_0 =\left\{\begin{array}{ll}
      b-\frac{\lambda}{\lambda_1},\ \ &{\rm for}\ \lambda \geq 0,
    \\[1.2ex]
   b , \ \ &{\rm for}\ \lambda < 0,
      \end{array}\right.
   \ \ {\rm and}\ \ c_1 =\left\{\begin{array}{ll}
      b,\ \ &{\rm for}\ \lambda \geq 0,
    \\[1.2ex]
   b -\frac{\lambda}{\lambda_1}, \ \ &{\rm for}\ \lambda < 0.
      \end{array}\right.
 \end{eqnarray}
 We have $b_0 > 0$ for $\lam < b\lam_1$, and $b_0 = 0$ for $\lam = b\lam_1$.\par
   If $\lam >b\lam_1$, we find that $b\|\nabla \psi_1\|_2^2 - \lambda \|\psi_1\|_2^2 = (b\lam_1 - \lam)\|\psi_1\|_2^2 <0$.\par
   We define
 $$L^+ = \left\{ u \in H^1_0(\Omega):~b\|\nabla u\|_2^2 - \lambda \|u\|_2^2 >0\right\},
 $$
 and define $L^0$, $L^-$ similarly by replacing `$>0$' by `$=0$' or `$<0$' respectively.\par
  $\bullet$ When $\lam <b\lam_1$, it follows from \eqref{1.6} that $L^0 = \left\{0\right\}$ and $L^-$ is empty. \par
  $\bullet$  When $\lam = b\lam_1$, we have $L^0 = \left\{u:~u=\tau \psi_1,\ \tau\in \mathbb{R}\right\}$ and $L^-$ is empty.\par
   $\bullet$ When $\lam >b\lam_1$, we deduce that $\tau \psi_1 \in L^-$ for all $\tau \in \mathbb{R}\backslash \{0\}$, and the size of $L^-$ gets bigger and bigger with increasing of $\lam$.\smallskip

  Let  $\Omega \subset \mathbb{R}^3$ be a bounded domain, and
  \begin{eqnarray}\label{1.8}
    S_q = \inf\limits_{u\in H^1_0(\Omega)\backslash \{0\}} \frac{\|\nabla u\|_2^2}{\|u\|^2_q}\ \ \ \ ({\rm for}\ 1<q\leq 6)
  \end{eqnarray}
  be the optimal Sobloev constant for the embedding $H^1_0(\Omega)\hookrightarrow L^q(\Omega)$, then $S_q >0$ and the embedding $H^1_0(\Omega)\hookrightarrow L^q(\Omega)$ is compact for $1< q< 6$.\par
   For simplicity, we write $\Lambda = S_4^2$, which will be used frequently in studying \eqref{zh809-1} in the case $p=3$. Because of the compactness of Sobolev embedding $H^1_0(\Omega)\hookrightarrow L^4(\Omega)$ and Fatou's lemma, then
   \begin{eqnarray*}
   \Lambda = \inf\left\{\|\nabla u\|^4_2 :~ u\in H^1_0(\Omega),\ {\rm and}\ \|u\|_4 =1\right\}>0
   \end{eqnarray*}
   is achieved by some $\phi_\Lambda \in H^1_0(\Omega)$ satisfying
    \begin{eqnarray}\label{1.9}
       \|\phi_\Lambda\|_4 =1, \ \ {\rm and}\ \  \|\nabla \phi_\Lambda\|_2^4 = \Lambda.
    \end{eqnarray}

  We define the energy corresponding to \eqref{zh809-1} by
   $$E(t)= \frac{1}{2}\int_{\Omega}|u_t|^2 \dif x + J(u),
   $$
  where
  \begin{eqnarray}\label{zh809-4}
  J(u)= \frac{a}{4}\|\nabla u\|_2^4 + \frac{b}{2}\|\nabla u\|_2^2 - \frac{\lambda}{2}\|u\|_2^2 - \frac{1}{p+1}\|u\|_{p+1}^{p+1}.
  \end{eqnarray}
 It is worthy to point out that the behavior of the solution to \eqref{zh809-1} relies heavily on the size of initial energy
 \begin{eqnarray}\label{1.11}
 E(0)= \frac{1}{2}\int_{\Omega}|u_1|^2 \dif x + J(u_0).
 \end{eqnarray}
   Consider the Nehari functional $I: H^1_0(\Omega)\rightarrow \mathbb{R}$ defined by
   \begin{eqnarray}\label{1.12}
     I(u)= \lag J'(u),~u\rag = a\|\nabla u\|_2^4 + b\|\nabla u\|_2^2 - \lambda \|u\|_2^2 - \|u\|_{p+1}^{p+1},
   \end{eqnarray}
where $J'$ is the Gateaux derivative of the functional $J$, and $\lag \cdot, \cdot \rag$ is the usual duality between $H^{-1}(\Omega)$ and $H^1_0(\Omega)$. Combining \eqref{1.6} with \eqref{1.12} and by a direct computation, we have
  \begin{eqnarray}\label{1.13}
    J(u) &=&a\left(\frac{1}{4}-\frac{1}{p+1}\right)\|\nabla u\|^4_2 + b\left(\frac{1}{2}-\frac{1}{p+1}\right)\|\nabla u\|^2_2\nonumber \\
     &~&- \lambda \left(\frac{1}{2}-\frac{1}{p+1}\right)\|u\|^2_2 +\frac{1}{p+1} I(u),
  \end{eqnarray}
and
  \begin{eqnarray}\label{1.14}
    E(0) &=& \frac{1}{2}\|u_1\|^2_2 + \frac{a(p-3)}{4(p+1)} \|\nabla u_0\|^4_2 + \frac{b(p-1)}{2(p+1)} \|\nabla u_0\|^2_2
       - \frac{\lambda (p-1)}{2(p+1)} \|u_0\|^2_2  + \frac{1}{p+1} I(u_0)\nonumber \\
       &\geq& \frac{1}{2}\|u_1\|^2_2 + \frac{a(p-3)}{4(p+1)} \|\nabla u_0\|^4_2 + \frac{(p-1)b_0}{2(p+1)} \|\nabla u_0\|^2_2 + \frac{1}{p+1} I(u_0).
  \end{eqnarray}

   In what follows, we introduce some function spaces which are of great used in our proof. For $p \geq 3$, we define
 \begin{eqnarray}\label{1.15}
 \textit{\textbf{N}}_p ~&=& \left\{ u \in H^1_0(\Omega)\backslash \{0\}:~I(u)=0\right\} \ \ \ \ -{\rm {\bf{the\ Nehari\ manifold}},}\\
 \textit{\textbf{N}}_p^{\,+} &=& \left\{ u \in H^1_0(\Omega):~I(u)>0\right\} \cup \{0\}\ \ -{\rm {\bf{inside\ space}},} \nonumber\\
 \textit{\textbf{N}}_p^{\,-} &=& \left\{ u \in H^1_0(\Omega):~I(u)<0\right\} \ \ \ \ \qquad -{\rm {\bf{outside\ space}},} \nonumber \\
 \textit{\textbf{W}}_{p} ^{\, +} &=& J^{d_p} \cap \textit{\textbf{N}}_p^{\,+} \ \ \quad \qquad\qquad\qquad -{\rm {\bf{stable\ set\ (potential\ well)}},} \nonumber \\
 \textit{\textbf{W}}_p^{\,-} &=& J^{d_p} \cap \textit{\textbf{N}}_p^{\,-} \ \ \quad \qquad\qquad\qquad -{\rm {\bf{unstable\ set}},} \nonumber
 \end{eqnarray}
 where
 \begin{eqnarray}\label{1.16}
    d_p = \inf\limits_{u \in \textit{\textbf{N}}_p} J(u)
  \end{eqnarray}
is called the depth of the potential well, and $J^{d_p}= \left\{ u \in H^1_0(\Omega): J(u)< d_p\right\}$
  is the sublevel set of $J$. Clearly, $0\not\in {\textit{\textbf{N}}_p} \cup {\textit{\textbf{N}}_p^{\,-}}$. \par

  It is turned out that Nehari manifold is a powerful tool in studying the stationary solutions to \eqref{zh809-1}. We refer the reader to the works of \cite{BZ, IO} concerning the structure of the Nehari manifold of \eqref{zh809-1} for the local case $a=0$, $1<p<5$ and nonlocal case $a>0$, $p=3$, $\lam =0$ respectively. For $a>0$, $1<p<5$ and $\lam \neq 0$, many interesting results on the Nehari manifold of \eqref{zh809-1} with $\lam u$ replaced by $\lam |u|^{q-1}u$ for $0<q<1$ are found in \cite{CKW}.

  Our focus here is to deal with the linear perturbation case. We have the following proposition to collect some properties concerning the geometry of ${\textit{\textbf{N}}_p}$, $\textit{\textbf{W}}_p^{\,+}$ and $\textit{\textbf{W}}_p^{\,-}$, which play important roles in studying the dynamics of \eqref{zh809-1}.

  \begin{prop}~\par
   {\rm\bf (1)} If $p=3$, $0<a<1/\Lambda$ and $\lambda < b\lambda_1$, then $\textit{\textbf{N}}_3$, $\textit{\textbf{W}}_3^{\,+}$, $\textit{\textbf{W}}_3^{\,-}$ are nonempty, and\par
  {\rm (i)} for every $u\in S$, where $S = \left\{u\in H^1_0(\Omega):\, \|u\|_4^4 - a \|\nabla u\|_2^4 > 0 \right\}$ {\rm(}see {\rm \eqref{3.5})}, there exists a unique number $\sigma_u >0$ {\rm(}see {\rm \eqref{3.7})}, such that $\sigma_u u \in \textit{\textbf{N}}_3$ satisfying $J(\sigma_u u) = \sup\limits_{\tau >0} J(\tau u)$, and $\tau u \in \textit{\textbf{N}}_3^{\,+}$ for $\tau \in (0, \sigma_u)$, $\tau u \in \textit{\textbf{N}}_3^{\,-}$ for $\tau \in (\sigma_u, +\infty)${\rm ;}\par
  {\rm (ii)} for any $u\in S^c$, we have $\tau u \in \textit{\textbf{N}}_3^{\,+}$ for $\tau >0${\rm ;}\par
  {\rm (iii)} $J(u)$ is bounded below on $\textit{\textbf{N}}_3$, thus the number $d_3$ appearing in $\eqref{1.16}$ with $p=3$ is well defined. Moreover, we have $d_3 >0$.\par
   {\rm\bf (2)} If $p=3$, $0<a<1/\Lambda$ and $b\lambda_1 < \lambda < b\lambda_1 + \delta$, where $\delta>0$ is determined by {\rm Lemma} $3.8$ below, and let $\sigma_u >0$ be the number defined by \eqref{3.7}. Then$:$\par  {\rm(iv)} $\textit{\textbf{N}}_3 \cap L^0 = \emptyset$, $\textit{\textbf{N}}_3 \cap L^+$ and $\textit{\textbf{N}}_3 \cap L^-$ are nonempty$;$\par
   {\rm (v)} for every $u\in S$, we have $\sigma_u u \in \textit{\textbf{N}}_3 \cap L^+$, $\tau u \in \textit{\textbf{N}}_3^{\,+}$ for $\tau \in (0, \sigma_u)$ and $\tau u \in \textit{\textbf{N}}_3^{\,-}$ for $\tau \in (\sigma_u, +\infty)${\rm ,} \par
  {\rm(vi)} for any $u\in L^-$, we have $\sigma_u u \in \textit{\textbf{N}}_3 \cap L^-$, $\tau u \in \textit{\textbf{N}}_3^{\,-}$ for $\tau \in (0, \sigma_u)$ and $\tau u \in \textit{\textbf{N}}_3^{\,+}$ for $\tau \in (\sigma_u, +\infty)${\rm ;}\par
   {\rm (vii)} $J(u)$ is bounded below on $\textit{\textbf{N}}_3 \cap L^-$, thus the number $d_3^{\,-} = \inf\limits_{u\in \textit{\textbf{N}}_3 \cap L^-} J(u)$ is well defined. Furthermore, we have
   $d_3^{\,-} <0$.\par
    {\rm\bf (3)} If $3<p<5$ and $\lambda \leq b\lambda_1$, then $\textit{\textbf{N}}_p$, $\textit{\textbf{W}}_p^{\,+}$, $\textit{\textbf{W}}_p^{\,-}$ are nonempty, and\par
     {\rm (viii)} for every $u\in H^1_0(\Omega) \setminus \{0\}$, there exists a unique number $\tau_u >0$ such that $J(\tau_u u) = \sup\limits_{\tau >0} J(\tau u)$, $\tau_u u \in \textit{\textbf{N}}_p$, $\tau u \in \textit{\textbf{N}}_p^{\,+}$ for $\tau \in (0, \tau_u)$ and $\tau u \in \textit{\textbf{N}}_p^{\,-}$ for $\tau \in (\tau_u, +\infty)${\rm ;}\par
  {\rm (ix)} $J(u)$ is bounded below on $\textit{\textbf{N}}_p$ and the number $d_p$ appearing in $\eqref{1.16}$ with $3<p<5$ is well defined. Moreover, we have $d_p > 0$.
   \end{prop}
  The detailed statement and proof of Proposition 1.1: {\bf (1)} are contained in Proposition 3.2, Proposition 3.5 and Theorem 3.7, {\bf (2)} are contained in Proposition 3.9, Lemma 3.10 and Lemma 3.11, and {\bf (3)} are contained in Proposition 4.2, Proposition 4.5 and Theorem 4.6. Our results also reveal some interesting relations between the value of $a$, $b$, $\lam$ and the size of $\textit{\textbf{N}}_p^{\,+}$, $\textit{\textbf{N}}_p^{\,-}$ for $3\leq p <5$; see Remark 3.3 and Remark 4.3. \par

  To conclude this subsection, we list some facts which will be used frequently in proving the main results.

  (i)  If $3\leq p<5$ and $\lambda \leq b\lambda_1$, thanks to \eqref{1.7} and \eqref{1.14} we have
  \begin{eqnarray}\label{1.17}
    \begin{array}{l}
       E(0)\geq 0 \quad {\rm for\ any}\ u_0 \in \textit{\textbf{N}}_{p} ^{\, +}\ {\rm and}\  u_1 \in L^2(\Omega),\\
       E(0)> 0 \quad {\rm for\ any}\ u_0 \in \textit{\textbf{N}}_{p} ^{\, +}{\rm ,}\  u_1 \in L^2(\Omega)\ {\rm and}\ \int_\Omega u_0 u_1 \dif x > 0.
    \end{array}
  \end{eqnarray}

  (ii) If $p=3$, $\lam \leq b\lam_1$ and $a \geq 1/\Lambda$, then recalling \eqref{1.12} and the definition of $\Lambda$, we have $I(u_0) \geq 0$ and $E(0) \geq 0$ for any $u_0 \in H_0^1(\Omega)$.\par
  (iii) If $\textit{\textbf{N}}_3 \cup \textit{\textbf{N}}_3^{\,-} \neq \emptyset$ and $\lam < b\lam_1$, then $a<1/\Lambda$, and
   \begin{eqnarray}\label{1.18}
    a\|\nabla u\|_2^4 - \|u\|_{4}^{4}  <0\quad {\rm for}\  u\in {\textit{\textbf{N}}_3} \cup {\textit{\textbf{N}}_3^{\,-}} .
   \end{eqnarray}

   In fact,for $u\in {\textit{\textbf{N}}_3} \cup {\textit{\textbf{N}}_3^{\,-}}$, we have $\|\nabla u\|_2 \neq 0$ and by \eqref{1.6},
   \begin{eqnarray*}
    a\|\nabla u\|_2^4 - \|u\|_{4}^{4} = I(u) - \left(b\|\nabla u\|_2^2 - \lambda \|u\|_2^2\right) \leq - b_0 \|\nabla u\|_2^2 <0.
   \end{eqnarray*}
  Thus \eqref{1.18} holds. Furthermore, noting that
  \begin{eqnarray}\label{1.19}
    \|u\|_{4}^{4} \leq \frac{1}{\Lambda} \|\nabla u\|_2^4,
  \end{eqnarray}
  then combing with \eqref{1.18} and \eqref{1.19}, we have
    \begin{eqnarray*}
     \left(a-\frac{1}{\Lambda}\right)\|\nabla u\|_2^4 \leq a\|\nabla u\|_2^4 - \|u\|_{4}^{4}  < 0,
    \end{eqnarray*}
   which leads to $a<1/\Lambda$ by virtue of $\|\nabla u\|_2 \neq 0$.   \hfill $\Box$

  \section{Main results}
  \setcounter{equation}{0}
  \label{} \ \ \ \ \   Recall the local existence result established by Ikehata and Okazawa \cite{IO}, we know that for any $\lam \in \mathbb{R}$, $p>1$ and the initial data $(u_0, u_1)$ satisfying \eqref{1.2}, the problem \eqref{zh809-1} admits a unique solution $u\in H(T)$ for some $T>0$, where
  \begin{eqnarray}\label{2.1}
      H(T) = C([0,T);H_0^1(\Omega)\cap H^2(\Omega)) \cap C^1([0,T);H_0^1(\Omega)) \cap C^2([0,T);L^2(\Omega)).
  \end{eqnarray}
    Furthermore, if the maximal existence time $T_{{\rm{max}}}$ of the solution $u = u(t)$ for \eqref{zh809-1} is finite, where we denote the function $u(t,x)$ by $u(t)$ for simplicity, then
   \begin{eqnarray*}
  \lim\limits_{t\rightarrow T_{{\rm{max}}}-} \left[\|\Delta u(t)\|_2 + \|\nabla u_t(t)\|_2 \right] = +\infty,
   \end{eqnarray*}
  and we say that $u$ blows up at time $T_{{\rm{max}}}$.

\subsection{Invariant sets under the flow of \eqref{zh809-1}}
 \label{} \ \ \ \ \  At first, we prove the invariance of some sets under the flow of \eqref{zh809-1}, which plays an important role in showing the blow-up behavior of the solution for \eqref{zh809-1} with $3\leq p<5$.
\begin{thm}
   Let $u \in H(T_{{\rm{max}}})$ be the unique solution of $\eqref{zh809-1}$ with initial data $(u_0, u_1)$ satisfying $\eqref{1.2}$, where $T_{{\rm{max}}}$ is the maximal existence time of $u(t)$.\\[0.3em]
  $\bullet$  Case A$:$ Let $p=3$, and assume that $0<a<1/\Lambda$, $\lam < b\lam_1$, $E(0)<d_3$, then we have{\rm :}\par
   $(A1)$ $u(t) \in \textit{\textbf{W}}_3^{\,+}$ for all $t \in [0, T_{{\rm{max}}}),$ provided that $u_0 \in \textit{\textbf{W}}_{3} ^{\, +}${\rm ;}\par
    $(A2)$ $u(t) \in \textit{\textbf{W}}_3^{\,-}$ for all $t \in [0, T_{{\rm{max}}}),$ provided that $u_0 \in \textit{\textbf{W}}_{3} ^{\, -}$.\\[0.3em]
  $\bullet$  Case B$:$ Let $p=3$, and assume that $0<a<1/\Lambda$, $b\lam_1 <\lam < b\lam_1 + \delta$, where $\delta >0$ is the number determined by Lemma $3.8$,
 then we have $u(t) \in \textit{\textbf{N}}_3^{\,-} \cap J^{d_3^{\,-}}$ for all $t \in [0, T_{{\rm{max}}}),$ provided that $u_0 \in \textit{\textbf{N}}_3^{\,-} \cap J^{d_3^{\,-}}$ and $E(0) < d_3^{\,-}$.\\
 $\bullet$  Case C$:$ Let $3<p<5$, and assume that $\lam \leq b\lam_1$, $E(0)< d_p$, then we have$:$\par
  $(C1)$ $u(t) \in \textit{\textbf{W}}_p^{\,+}$ for all $t \in [0, T_{{\rm{max}}}),$ provided that $u_0 \in \textit{\textbf{W}}_{p} ^{\, +}${\rm ;}\par
   $(C2)$ $u(t) \in \textit{\textbf{W}}_p^{\,-}$ for all $t \in [0, T_{{\rm{max}}}),$ provided that $u_0 \in \textit{\textbf{W}}_{p} ^{\, -}$.
 \end{thm}

   In fact, under the cases (A1) and (C1), we deduce from \eqref{1.17} that the initial energy $E(0)$ must be non-negative when $u_0$ lies within the potential well $\textit{\textbf{W}}_{p} ^{\, +}$.

 The proof of Theorem 2.1 relies heavily upon the behavior of $J(u)$ and geometry of $\textit{\textbf{N}}_p$, $\textit{\textbf{N}}_p^{\,+}$ and $\textit{\textbf{N}}_p^{\,-}$ for $3\leq p <5$, which are described by Lemma 3.4 and Proposition 3.5 for Case $A$, Lemma 3.11 for Case $B$ and Lemma 4.4, Proposition 4.5 for Case $C$ respectively.

\subsection{The 4-sublinear case}

 \label{} \ \ \ \ \  The following result asserts that the solution for \eqref{zh809-1} with $1<p<3$ is bounded uniformly in time and there exists a vacuum region such that \eqref{zh809-1} admits no solution in it.

 \begin{thm}
     $($Boundedness of the solution in the case $1<p<3$$)$ Suppose that $1<p<3$, $\lambda \in \mathbb{R}$, $a$, $b>0$, and let $u \in H(T_{{\rm{max}}})$ be a solution of $\eqref{zh809-1}$ with $(u_0, u_1)$ satisfying $\eqref{1.2}$, where $T_{{\rm{max}}}$ is the maximal existence time for $\eqref{zh809-1}$. Then, there exists a constant $K_1>0$ such that
       \begin{eqnarray}\label{2.2}
         \|u_t (t)\|_2 + \|\nabla u(t)\|_2 \leq K_1\qquad for\ all\ \ t\in [0, T_{{\rm{max}}}),
       \end{eqnarray}

     Moreover, if $a>0$ is sufficiently small, $\lam <b\lam_1$, and we assume that $E(0)\leq 0$, $\|\nabla u_0\|_2 >0$, then
     \begin{eqnarray}\label{2.3}
      \|\nabla u(t)\|_2 \geq \left(\frac{p+1}{2}b_0 S_{p+1}^{(p+1)/2}\right)^{1/(p-1)}\qquad for\ all\ \ t\in [0, T_{{\rm{max}}}),
     \end{eqnarray}
     where $b_0$ is defined by $\eqref{1.7}$.
   \end{thm}
 \begin{rmk}
 {\rm (1) The condition for $a>0$ to be sufficiently small plays an important role in proving the existence of initial data $(u_0, u_1)$ which lead to \eqref{2.3} (see Proposition 6.1).}

  {\rm (2)} Vacuum region of the solution$:$\par
      {\rm From \eqref{2.3}, we know that under the conditions of Theorem 1.1, there exists a vacuum region
      $$U_{\lam} = \left\{u\in H_0^1(\Omega):~0< \|\nabla u\|_2 < \left(\frac{p+1}{2}b_0 S_{p+1}^{(p+1)/2}\right)^{1/(p-1)}\right\},
      $$
   such that \eqref{zh809-1} admits no solution in $U_\lam$. In view of \eqref{1.7},
    we find that $U_\lam$ gets bigger and bigger as $\lam$ decreases, and as the limit case that $\lam \rightarrow 0$ we obtain
    $$U_{0} = \left\{u\in H_0^1(\Omega):~0< \|\nabla u\|_2 < \left(\frac{p+1}{2}b S_{p+1}^{(p+1)/2}\right)^{1/(p-1)}\right\}.
   $$

   We would like to mention that, in the case when $p\geq3$ we also obtain a vacuum region under some suitable assumptions on $a$, $b$, $\lam$ and the initial data of \eqref{zh809-1} (see  Theorem $2.6$ for $p=3$ and Theorem $2.10$ for $p>3$ respectively).\hfill $\Box$}
   \end{rmk}

\subsection{The asymptotically 4-linear case}

 \label{} \ \ \ \ \
    Next, we prove the following result concerning the boundedness and asymptotically behavior of the solution for \eqref{zh809-1} when $p=3$.

   \begin{thm}
     {\rm(}Boundedness of the solution in the case $p=3$, $\lam<b\lam_1)$ For $p=3$ and $b>0$, let $u \in H(T_{{\rm{max}}})$ be a solution of $\eqref{zh809-1}$ with $(u_0, u_1)$ satisfying $\eqref{1.2}$ and $\Lambda$ be the constant defined by \eqref{1.9}. Suppose that one of the following cases holds$:$\par
     $(H1)$ $a> 1/\Lambda$\ \  and\ \  $\lambda \in \mathbb{R}$,\par
     $(H2)$ $a= 1/\Lambda$\ \  and\ \  $\lambda < b\lambda_1$,\par
     $(H3)$ $0<a<1/\Lambda$,\ \ $\lambda < b\lambda_1$,\ \ $u_0 \in \textit{\textbf{W}}_{3} ^{\, +}$ and $E(0)<d_3$,\\
      then there exists a constant $K_2>0$ such that
       \begin{eqnarray}\label{2.4}
          \|u_t (t)\|_2 + \|\nabla u(t)\|_2 \leq K_2\qquad {\rm for}\ \ t\in [0, T_{{\rm{max}}}).
        \end{eqnarray}
 \label{} \ \ \ \   Furthermore, if $(H3)$ holds and assume in addition that
 \begin{eqnarray}\label{2.5}
      \| \psi_1\|_4^4 - a\|\nabla \psi_1\|_2^4 >0,
   \end{eqnarray}
 where $\psi_1$ is the positive principal eigenfunction associated with the principal eigenvalue $\lam_1$ of the Laplacian operator $-\Delta$,
  then we have
    \begin{eqnarray}\label{2.6}
  \|u_t (t)\|_2 + \|\nabla u(t)\|_2 \rightarrow 0,\ \ {\rm as}\ \lam \rightarrow b\lam_1, \quad {\rm for\ any}\ \ t\in [0, T_{{\rm{max}}}).
 \end{eqnarray}
   \end{thm}


  \begin{rmk}
    $(1)$ {\rm In Proposition 6.2, we prove the existence of initial data which satisfy the assumptions of Theorem 2.4 and give a sufficient condition which leads to \eqref{2.5}}.

$(2)$ {\rm The conditions of $\lam < b\lam_1$ and \eqref{2.5} play important roles in proving \eqref{2.6}. It is interesting to note that the behavior of the solution for \eqref{zh809-1} with $p=3$ depends heavily upon the sign of $\|\psi_1\|_4^4 - a\|\nabla \psi_1\|_2^4$ as $\lam$ varies. We consider the case when $\lam >b\lam_1$ and $\|\psi_1\|_4^4 - a\|\nabla \psi_1\|_2^4 <0$ in Theorem 2.8 below.} {\rm  \hfill $\Box$}
\end{rmk}

    Now we turn to the following two questions. We consider what happens when \par
    {\bf(i)} $\lam <b\lam_1$, $E(0)\not\in [0, d_3)$ or $u_0$ lies outside $\textit{\textbf{W}}_{3} ^{\, +}$ for $0<a<1/\Lambda$, $\lambda < b\lambda_1$,\par
    {\bf(ii)} $\lam >b\lam_1$ and $\|\psi_1\|_4^4 - a\|\nabla \psi_1\|_2^4 <0$.\par\medskip
      For (i), we obtain the blow-up property of the solution for \eqref{zh809-1} with arbitrary initial energy $E(0)$ and initial data starting in the unstable set.



  \begin{thm}
   {\rm(}Conditions for blow-up in the case $p=3$, $\lam <b\lam_1)$ Let $p=3$, $0<a<1/\Lambda$, $\lam<b\lam_1$, and let $u \in H(T_{{\rm{max}}})$ be a solution of $\eqref{zh809-1}$ with $(u_0, u_1)$ satisfying $\eqref{1.2}$. Suppose that
    \begin{eqnarray}\label{2.7}
      \left(\textit{\textbf{N}}_3 \cup \textit{\textbf{N}}_3^{\,-}\right) \cap H^2(\Omega) \neq \emptyset,
    \end{eqnarray}
    and one of the following cases holds$:$\par
    $(h1)$ $E(0)<0$,\par
    $(h2)$ $E(0)=0$ and assume in addition that
       \begin{eqnarray}\label{2.8}
       \int_\Omega u_0 u_1 \dif x > 0,
       \end{eqnarray}

    $(h3)$ $0< E(0) <d_3$\   and $u_0 \in \textit{\textbf{W}}_3^{\,-}$,\par
    $(h4)$ $E(0) \geq d_3$,\ \  $\eqref{2.8}$ holds,\  and
      \begin{eqnarray}\label{2.9}
      \|u_0\|_2^2 > \frac{4}{b_0\lambda_1}E(0),
      \end{eqnarray}
    where $b_0$ is the constant appearing in $\eqref{1.7}$.\par
   Then the maximal existence time $T_{{\rm{max}}}$ of $u$ is finite.
   Moreover, we have
   \begin{eqnarray}\label{2.10}
     \|\nabla u(t)\|_2^2 \geq \frac{2 b_0 \Lambda}{1- a \Lambda}, \quad for\ all \ \ t\in [0, T_{{\rm{max}}}),
   \end{eqnarray}
   provided that $E(0) \leq 0$ and $\|\nabla u_0\|_2 >0$.
  \end{thm}

    \begin{rmk}
      {\rm (1) The assumptions of Theorem 2.6 $(h1)$, $(h2)$ lead to the fact that $u_0 \in \textit{\textbf{W}}_3^{\,-}$, and $(h4)$ implies} $u_0 \in \textit{\textbf{N}}_3^{\,-}$.\par
  {\rm Indeed, if $(h1)$ or $(h2)$ satisfies, then by \eqref{1.11} and \eqref{1.14} for $p=3$, we obtain that $J(u_0) <0$ and $I(u_0)<0$. Hence $u_0 \in \textit{\textbf{W}}_3^{\,-}$.\par
     If $(h4)$ satisfies, then by \eqref{1.14} for $p=3$, we have
    $$I(u_0) \leq 4E(0) - b_0 \lam_1 \|u_0\|_2^2 <0,
    $$
  which implies} $u_0 \in \textit{\textbf{N}}_3^{\,-}$. \par
  {\rm  (2) For $p=3$ and suppose that $0<a<1/\Lambda$, $\lam <b\lam_1$, then \eqref{2.10} shows that \eqref{zh809-1} with $(u_0, u_1)$ satisfying \eqref{1.2}, $E(0)\leq 0$ and $\|\nabla u_0\|_2 >0$ admits no solution in the ball}
  $$ U= \left\{u\in H_0^1(\Omega):~\|\nabla u\|_2^2 < \frac{2 b_0 \Lambda}{1- a \Lambda}\right\}.
  $$

  {\rm(3) By Proposition 6.3 below, we know that the hypothesis of \eqref{2.7} ensures the existence of initial data which satisfy $\eqref{1.2}$ and each of the conditions of $(h1)$-$(h4)$ in Theorem $2.6$. If \eqref{2.7} does not hold, then for any initial data $(u_0, u_1)$ satisfying $\eqref{1.2}$, we deduce that $u_0 \in \textit{\textbf{N}}_3^{\,+}$. Hence, \eqref{2.7} is vital for us to study the blow-up phenomenon of \eqref{zh809-1}}.\hfill $\Box$
  \end{rmk}

 Next, we study the question (ii): the case when $\lam > b\lam_1$ and $\|\psi_1\|_4^4 - a\|\nabla \psi_1\|_2^4 <0$.  We have:
 \begin{thm}
   {\rm(}Conditions for blow-up in the case $p=3$, $\lam >b\lam_1)$ Let $p=3$, $0<a<1/\Lambda$ and suppose that
    \begin{eqnarray}\label{2.11}
     \| \psi_1\|_4^4 - a\|\nabla \psi_1\|_2^4 &<& 0,\end{eqnarray}
      \begin{eqnarray}\label{2.12}
       \left(L^+ \cap \textit{\textbf{N}}_3 \right) \cap H^2(\Omega) &\neq& \emptyset.
    \end{eqnarray}
    Let $u \in H(T_{{\rm{max}}})$ be a solution of $\eqref{zh809-1}$ with $(u_0, u_1)$ satisfying $\eqref{1.2}$, $u_0 \in \textit{\textbf{N}}_{3} ^{\, -}$ and $E(0) <d_3^-$, where
    \begin{eqnarray}\label{2.13}
      d_3^- = \inf\limits_{u \in \textit{\textbf{N}}_3 \cap L^-} J(u).
    \end{eqnarray}
       Then there exists $\delta >0$ such that for any $\lam \in (b\lam_1,~b\lam_1 + \delta)$, the maximal existence time $T_{{\rm{max}}}$ of $u$ is finite.
    \end{thm}


   We should note that, when $\lam >b\lam_1$, the structure of the Nehari manifold $\textit{\textbf{N}}_3$ and the behavior of $J(u)$ on $\textit{\textbf{N}}_3 \cap L^-$ relies heavily on the assumption that $ \| \psi_1\|_4^4 - a\|\nabla \psi_1\|_2^4 < 0$ (see Proposition 3.9 and Lemma 3.10 for details). \par

   Moreover, Proposition 6.4 in Section 6 shows that \eqref{2.12} ensures the existence of initial data which satisfy the assumptions of Theorem 2.8, and gives a sufficient condition which leads to \eqref{2.12}.\par


\subsection{The 4-superlinear case}
 \label{} \ \ \ \ \
   Finally, we turn to the 4-superlinear problem.\par
  Our first result in this direction claims that the solution to \eqref{zh809-1} is bounded uniformly in time provided that $3<p<5$, $\lam \leq b\lam_1$ with initial data starting in the stable set.
   \begin{thm}
     {\rm(}Boundedness of the solution in the case $3<p<5)$ For $3<p<5$ and $a$, $b>0$, let $u \in H(T_{{\rm{max}}})$ be a solution of $\eqref{zh809-1}$ with $(u_0, u_1)$ satisfying $\eqref{1.2}$. Suppose that $\lam \leq b\lam_1$, $u_0 \in \textit{\textbf{W}}_{p} ^{\, +}$ and $E(0)<d_p$, then there exists a constant $K_3>0$ such that
       \begin{eqnarray}\label{2.14}
          \|u_t (t)\|_2 + \|\nabla u(t)\|_2 \leq K_3\qquad for\ \ t\in [0, T_{{\rm{max}}}).
        \end{eqnarray}
           \end{thm}
  Proposition 6.5 shows the existence of initial data satisfying the conditions of Theorem 2.9.

   If the conditions of Theorem 2.9 do not hold, then the solution of \eqref{zh809-1} may blow-up at a finite time for suitable initial data and the parameter $\lam$ lying in some certain intervals. We have

 \begin{thm}
   {\rm (}Blowing up solutions for $p>3)$ Let $p>3$, and $u \in H(T_{{\rm{max}}})$ be a solution to $\eqref{zh809-1}$ with $(u_0, u_1)$ satisfying $\eqref{1.2}$. Assume that one of the following assumptions satisfies{\rm :}\\
$\bullet$  Case $(a):$ If $\lambda \leq b\lambda_1$, and we assume that :\par
  $(a1)$ $E(0)<0$, or\par
  $(a2)$ $E(0)=0$ and $\eqref{2.8}$ hold, or\par
  $(a3)$ $3<p<5$, $0< E(0) <d_p$, $\eqref{2.8}$ hold and $u_0 \in \textit{\textbf{W}}_{p} ^{\, -}$, or\par
   $(a4)$ $E(0) >0$, $\eqref{2.8}$ hold and
      \begin{eqnarray}\label{2.15}
      \|u_0\|_2^4 > \frac{4(p+1)}{(p-3)a\lam_1^2}E(0).
      \end{eqnarray}
$\bullet$ Case $(b):$ If $\lambda> b\lambda_1$, and we assume that$:$\par
   $(b1)$ $E(0) < -h_0/(2p+2)$, or \par
   $(b2)$ $E(0) = -h_0/(2p+2)$ and $\eqref{2.8}$ hold, where
  \begin{eqnarray}\label{2.16}
      h_0 = \frac{(p-1)^2 (\lam - b\lam_1)^2 }{2(p-3)a\lam_1^2}.
  \end{eqnarray}
  Then the maximal existence time $T_{{\rm{max}}}$ of $u$ is finite.\\[0.3em]
$\bullet$ Vacuum region of the solution$:$ If $3<p<5$ and $\lam \leq b\lam_1$, then
\begin{eqnarray}\label{2.17}
    \|\nabla u(t)\|_2 \geq \left[\frac{p+1}{4}aS_{p+1}^{(p+1)/2}\right]^{1/(p-3)}, \quad for\ all \ \ t\in [0, T_{{\rm{max}}}),
\end{eqnarray}
   provided that $E(0) \leq 0$ and $\|\nabla u_0\|_2 >0$.
  \end{thm}
 \begin{rmk}
  {\rm (1) The existence of initial data satisfying the assumptions of Theorem 2.10 is proved in Proposition 6.6.}\par
  {\rm (2) Contrast to the asymptotically 4-linear case (see Theorem 2.6 and Theorem 2.8), we find that the coefficient $a$ may not  qualitatively affect the blow-up behavior of the solution for \eqref{zh809-1} when $p>3$.}\hfill $\Box$
  \end{rmk}
The Table 1 below summarizes our results concerning with the boundedness and blow-up of solutions for \eqref{zh809-1}, and the Table 2 lists the assumptions on $p$, $a$, $b$, $\lam$ and the initial data to generate the invariant set and vacuum region of the solution for \eqref{zh809-1}.
\begin{table}[H]
{ \bf \caption{Our results about the boundedness and blow-up of solutions for \eqref{zh809-1}}}\par\medskip

\centering
    \begin{tabular}{|c|c|c|c|}
  \bottomrule
   $p$ & Range of $a$, $b$, $\lambda$ & Assumptions of initial data & Result  \\
  \hline
    \multirow{2}*{$1<p<3$} &  \multirow{2}*{$a>0$, $b>0$, $\lambda \in \mathbb{R}$} &  \multirow{2}*{\eqref{1.2}} & bounded, \\ {} & {}& {} & Theorem 2.2  \\
  \hline
   \multirow{9}*{$p=3$} & $a>\frac{1}{\Lambda}$, $b>0$, $\lambda \in \mathbb{R}$ & \multirow{2}*{\eqref{1.2}} &\multirow{2}* {bounded,}   \\
  \cline{2-2}
  {} & $a=\frac{1}{\Lambda}$, $\lam < b\lam_1$ & {} & \multirow{2}*{Theorem 2.4}\\
  \cline{2-3}
  {} &  \multirow{5}*{$0<a<\frac{1}{\Lambda}$, $\lam < b\lam_1$} & \eqref{1.2}, $u_0 \in \textit{\textbf{W}}_{3} ^{\, +}$, $E(0)<d_3$ & {}\\
  \cline{3-4}
  {} & {} & \eqref{1.2}, $E(0)<0$ & \multirow{2}*{blow-up,}\\
  \cline{3-3}
   {} & {} & \eqref{1.2}, \eqref{2.8}, $E(0)=0$ & {}\\
  \cline{3-3}
   {} & {} & \eqref{1.2}, $u_0 \in \textit{\textbf{W}}_{3} ^{\, -}$, $0<E(0)<d_3$ & \multirow{2}*{Theorem 2.6}\\
  \cline{3-3}
   {} & {} & \eqref{1.2}, \eqref{2.8}, \eqref{2.9}, $E(0)\geq d_3$ & {}\\
  \cline{2-4}
  {} & $0<a<\frac{1}{\Lambda}$,  & \eqref{1.2}, (2.11), (2.12),   & blow-up, \\
  {} & $b\lam_1 < \lam < b\lam_1 + \delta$ & $u_0 \in \textit{\textbf{N}}_{3} ^{\, -}$, $E(0)<d_3^-$ & Theorem 2.8\\
  \hline
  \multirow{4}*{$3<p<5$} & \multirow{4}*{$a>0$, $\lam \leq b\lambda_1$} & \multirow{2}*{\eqref{1.2}, $u_0 \in \textit{\textbf{W}}_{p} ^{\, +}$, $E(0)<d_p$} &  bounded,   \\
  {} & {} & {}&  Theorem 2.9  \\
  \cline{3-4}
  {} & {} & \eqref{1.2}, \eqref{2.8}, $u_0 \in \textit{\textbf{W}}_{p} ^{\, -}$,  &  \multirow{5}*{blow-up, }\\
   {} & {} &   $0<E(0)<d_p$ &  {}\\
  \cline{1-3}
  \multirow{5}*{$p>3$} & \multirow{3}*{ $a>0$, $\lambda \leq b\lambda_1$} & \eqref{1.2}, $E(0)<0$ & {}\\
  \cline{3-3}
  {} & {} & \eqref{1.2}, \eqref{2.8}, $E(0)=0$ & \multirow{3}*{Theorem 2.10}\\
  \cline{3-3}
   {} & {} & \eqref{1.2}, \eqref{2.8}, \eqref{2.15}, $E(0)>0$ & {}\\
  \cline{2-3}
  {} & \multirow{2}*{ $a>0$, $\lambda > b\lambda_1$} & \eqref{1.2}, $E(0) < -h_0 / (2p+2)$ & {}\\
  \cline{3-3}
  {} & {} & \eqref{1.2}, \eqref{2.8}, $E(0) = -h_0 / (2p+2)$ & {}\\
  \toprule
\end{tabular}
\end{table}

\begin{table}[H]
 \centering
 {\bf \caption{Sufficient conditions for the existence of}}
 {\centering\bf{invariant set or vacuum region of the solution}}\par\medskip

    \begin{tabular}{ccc}
  \bottomrule
 \multirow{2}*{ Value of $p$} & Assumptions to generate the &  Assumptions to generate the    \\
  {} & invariant set of the solution & vacuum region of the solution \\
  \bottomrule
    \multirow{3}*{$1<p<3$} & \multirow{3}*{no result} & $a>0$ is sufficiently small,  \\
     {} & {} & $\lam <b\lam_1$, $E(0)\leq 0$, $\|\nabla u_0\|_2 >0$    \\
     {} & {} & (Theorem 2.2)\\
  \hline
   \multirow{3}*{$p=3$} & $0<a<\frac{1}{\Lambda}$, $\lam < b\lam_1$ and $E(0)< d_3$ & $0<a<\frac{1}{\Lambda}$, $\lam < b\lam_1$, \eqref{2.7},  \\
    {} &  (or $b\lam_1 < \lam < b\lam_1 + \delta$, $E(0)< d_3^-$) & $E(0)\leq 0$, $\|\nabla u_0\|_2 >0$  \\
    {} & (Cases $A$ and $B$ of Theorem 2.1)  & (Theorem 2.6)\\
  \hline
  \multirow{3}*{$3<p<5$} & $a>0$, $\lam \leq b\lam_1$, & $a>0$, $\lam \leq b\lam_1$,  \\
  {} & $E(0)<d_p$ & $E(0)\leq 0$, $\|\nabla u_0\|_2 >0$\\
  {} & (Case $C$ of Theorem 2.1) & (Theorem 2.10)\\
  \toprule
\end{tabular}

\end{table}

\subsection{Comments on our results}
 \label{} \ \ \ \ \
   1. \emph{Differences between the behavior of solutions for $\eqref{zh809-1}$ and $\eqref{1.3}$}\par
   It is worth noting that the conclusions of Theorem 2.2 and the cases $(H1)$, $(H2)$ of Theorem 2.4 reveal some interesting phenomena which present striking differences between the behavior of solutions of hyperbolic Kirchhoff type equation \eqref{zh809-1} (the case $a>0$) and that of wave equation \eqref{1.3} (the case $a=0$), as we know the solutions of \eqref{1.3} for $p>1$ may blow-up in finite time when the initial data of \eqref{1.3} lie within some unstable sets (see \cite{C, Liu, LZ}). While for $a>0$ and any initial data satisfying \eqref{1.2}, the solutions of \eqref{zh809-1} are bounded uniformly in time when (i) $1<p<3$, or (ii) $p=3$, $a \geq 1/\Lambda$ and $\lam<b\lam_1$. \par

   2. In \cite{Liu}, Liu studied the blow-up behavior of the solution for the local problem \eqref{1.3} with positive initial energy and initial data lying in the unstable set under the assumption of \eqref{2.8}. Theorem 2.6 $(h3)$ concludes that the condition \eqref{2.8} is not necessary to obtain the blow-up result for \eqref{zh809-1} due to the nonlocal effect.

   3. For $p=3$ and $\lam = 0$, the blow-up phenomena of \eqref{zh809-1} was studied by \cite{I, IO} for initial energy $E(0)$ below the depth of the potential well. Our work demonstrates the connection between the blow-up of problem \eqref{zh809-1} and the value of $p$, $\lam$. We generalize the results of \cite{I, IO} in the following four aspects:\par
   (i) In Theorem 2.6 we discuss the case of $p=3$, $0<a<1/\Lambda$ and $\lambda < b\lambda_1$ with arbitrary initial energy $E(0)$;\par
    (ii) In Theorem 2.8 we construct the the finite time blow-up solution for \eqref{zh809-1} with low energy when $\lam$ lies in some interval $(b\lam_1, b\lam_1 + \delta)$; \par
    (iii) In Theorem 2.10 we study the blow-up of \eqref{zh809-1} with $p>3$ and for both cases $\lambda \leq b\lambda_1$ and $\lambda > b\lambda_1$;\par
    (iv) We build the conditions of the initial data to obtain the vacuum region of the solution for \eqref{zh809-1} with $1<p<5$.


\section{Structure of Nehari manifold $\textit{\textbf{N}}_3$}
\setcounter{equation}{0}
\label{}\ \ \ \ \
  It is well known that the stationary solution for \eqref{zh809-1} correspond to the critical point of the energy functional $J$ which is defined by \eqref{zh809-4}. Due to \eqref{1.12} and \eqref{1.15}, we can see that all critical points of $J$ must lie on the Nehari manifold $\textit{\textbf{N}}_p$.\par
 For $3\leq p <5$, the nature of $\textit{\textbf{N}}_p$ connects strongly with the behavior of fiber mapping $K_u : \mathbb{R}^+ \rightarrow \mathbb{R}$ in the form of
  \begin{eqnarray}\label{3.1}
    K_u(\tau) = J(\tau u) = \frac{a \tau^4}{4} \|\nabla u\|_2^4 + \frac{b \tau^2}{2} \|\nabla u\|_2^2 - \frac{\lambda \tau^2}{2}  \|u\|_2^2 - \frac{\tau^{p+1}}{p+1} \|u\|_{p+1}^{p+1},\quad  \forall~\tau >0 ,
  \end{eqnarray}
 for every $u\in H^1_0(\Omega)$ and $3\leq p <5$. Such a map was studied by Brown and Zhang \cite{BZ} for the case $a=0$, Ikehata and Okazawa \cite{IO} for the case $a>0$, $p=3$, $\lam=0$.

 We have
  \begin{eqnarray}\label{3.2}
    K_u'(\tau)&=& a\tau^3 \|\nabla u\|_2^4 + b\tau \|\nabla u\|_2^2 - \lambda \tau \|u\|_2^2 - \tau^{p} \|u\|_{p+1}^{p+1},
    \end{eqnarray}
    and
     \begin{eqnarray}\label{3.3}
     K_u''(\tau)&=& 3a\tau^2 \|\nabla u\|_2^4 + b \|\nabla u\|_2^2 - \lambda \|u\|_2^2 - p \tau^{p-1} \|u\|_{p+1}^{p+1}.
  \end{eqnarray}
  It is easy to infer that $I(\tau u) = \tau K_u'(\tau)$ from \eqref{1.12} and \eqref{3.2}. So, for $u\in H^1_0(\Omega)\backslash \{0\}$ and $\tau >0$, we have $K_u'(\tau) = 0$ if and only if $\tau u \in \textit{\textbf{N}}_p$, which means that elements in $\textit{\textbf{N}}_p$ correspond to stationary points of the fiber mapping $K_u (\tau)$. In particular, we obtain that $K_u'(1) = 0$ if and only if $u \in \textit{\textbf{N}}_p$.\par\smallskip

 The nature of $\textit{\textbf{N}}_p$ is closely related to the value of $p$, $\lam$. In this section, we investigate the Nehari manifold $\textit{\textbf{N}}_3$ via the fiber mapping method, and discuss how the structure of $\textit{\textbf{N}}_3$ changes as $\lam$ changes.

  \subsection{The case of $p=3$, $\lam < b\lam_1$}
3.1.1 \emph{Fiber mapping analysis and the structure of $\textit{\textbf{N}}_3$ for $\lam < b\lam_1$}\par\medskip
Since $\lam < b\lam_1$, as mentioned in introduction, we deduce from \eqref{1.6} that $L^0 = \left\{0\right\}$ and $L^-$ is empty, which imply that $b\, \|\nabla u\|_2^2 - \lambda\, \|u\|_2^2 >0$ for any $u\in H_0^1(\Omega)\backslash \{0\}$, and $b\, \|\nabla u\|_2^2 - \lambda\, \|u\|_2^2 =0$ if and only if $\|\nabla u\|_2 =0$. In view of these facts and \eqref{3.1} when $p=3$, we know that the behavior of fiber mapping $K_u$ for $p=3$ depends on the sign of $\|u\|_4^4 - a \|\nabla u\|_2^4$.

 To show this, we define
   \begin{eqnarray}\label{3.4}
     h_3(\tau) = \left(a \|\nabla u\|_2^4 - \|u\|_4^4 \right)\tau^2 + b\, \|\nabla u\|_2^2 - \lambda\, \|u\|_2^2.
   \end{eqnarray}
 If $u$ satisfies $\|u\|_4^4 - a \|\nabla u\|_2^4 < 0$, then $h_3(\tau)$ is increasing in $\tau \in [0, +\infty)$ and tends to $+\infty$ as $\tau\rightarrow +\infty$, which implies that $K_u(\tau)$ does not have any stationary point.\par
  In order to obtain a $\tau >0$ such that $K'_u(\tau) =0$, we introduce a subspace of $H^1_0(\Omega)$ as following:
   \begin{eqnarray}\label{3.5}
     S = \left\{u\in H^1_0(\Omega):\, \|u\|_4^4 - a \|\nabla u\|_2^4 > 0 \right\}.
   \end{eqnarray}
It is obviously that $0\not\in S$. Further we claim:\par
  (i) $S \neq \emptyset$ assuming that $0<a<1/\Lambda$, where $\Lambda>0$ is the constant appearing in \eqref{1.9};\par
  (ii) If  $u\in S$, then $\tau u \in S$ for $\tau >0$.\par
   In fact, by virtue of \eqref{1.9} and $0<a<1/\Lambda$, we have
   \begin{eqnarray}\label{3.6}
   a \|\nabla \phi_\Lambda\|_2^4 - \|\phi_\Lambda\|_4^4 = a\Lambda -1 <0
   \end{eqnarray}
    and $\phi_\Lambda \in S$. Hence (i) is obtained. It is easy to observe from \eqref{3.5} that (ii) holds.

  The following lemma shows the nature of the fiber mapping $K_u$ for $p=3$.

\begin{lem}
  Let $p=3$, $\lambda< b\lambda_1$ and $0<a<1/\Lambda$. Then for every $u\in S$, it holds that\par
 {\rm (i)} there is a unique number $\sigma_u>0$ depending on $u$, such that $K_u'\left(\sigma_u\right) = 0${\rm ;}\par
 {\rm (ii)} $K_u'\left(\tau\right) > 0$ for $0< \tau< \sigma_u${\rm,} and $K_u'\left(\tau\right) < 0$ for $\tau > \sigma_u${\rm ;}\par
 {\rm (iii)} $K_u''\left(\sigma_u\right) < 0$, and $K_u(\sigma_u) =  \sup\limits_{\tau >0}K_u(\tau)$.\par
 While, for $u\in S^c= \left\{u\in H^1_0(\Omega):\, a \|\nabla u\|_2^4 - \|u\|_4^4 \geq 0\right\}$ and $\|\nabla u\|_2\neq 0$, we have\par
 {\rm (iv)} $K_u'\left(\tau\right) > 0$ for $\tau >0$, and  $\sup\limits_{\tau >0}K_u(\tau) = +\infty$.
\end{lem}
\begin{proof}
  (i) For $u\in S$, by $\lambda< b\lambda_1$ and \eqref{1.6}, then we obtain a unique number
  \begin{eqnarray}\label{3.7}
    \sigma_{u} = \left(\frac{ b\|\nabla u\|_2^2 - \lambda \|u\|_2^2}{\|u\|_{4}^{4} - a\|\nabla u\|_2^4}\right)^{\frac{1}{2}} >0
  \end{eqnarray}
  such that
  \begin{eqnarray*}
    h_3(\sigma_{u}) = \left(a \|\nabla u\|_2^4 - \|u\|_4^4 \right)\sigma_{u}^2 + b\, \|\nabla u\|_2^2 - \lambda\, \|u\|_2^2 = 0.
  \end{eqnarray*}
  Hence $K_u'(\sigma_u)=\sigma_u h_3 (\sigma_u) = 0$.\par
  (ii) It follows from \eqref{3.4}, \eqref{3.5} and \eqref{3.7} that $h_3 (\tau) > 0$ for $0\leq \tau< \sigma_u$, and $h_3 (\tau) < 0$ for $\tau > \sigma_u$. Further, noting that $K_u'(\tau)=\tau h_3 (\tau)$, we conclude (ii).\par
  (iii) Using \eqref{3.7}, \eqref{3.3} for $p=3$ and \eqref{1.6}, we have
  $$K_u''\left(\sigma_u\right) = 3 \left (a \|\nabla u\|_2^4 - \|u\|_4^4 \right) \sigma_u^2 + b\|\nabla u\|_2^2 - \lambda \|u\|_2^2 = -2 \left(b\|\nabla u\|_2^2 - \lambda\|u\|_2^2 \right) < 0,
  $$
for $u\in S$ and $\lambda< b\lambda_1$. Then $K_u(\sigma_u) =  \sup\limits_{\tau >0}K_u(\tau)$.\par
  (iv) For $u\in S^c$, by \eqref{1.6} and \eqref{3.4} we get $h_3(\tau)>0$, and $K_u'(\tau)=\tau h_3 (\tau)>0$ for every $\tau>0$. Moreover, if $u\in S^c$, then we deduce from \eqref{3.1} with $p=3$ that $K_u(\tau)\geq \frac{1 }{2}  \left( b\|\nabla u\|_2^2 -  \lambda \|u\|_2^2\right)\tau^2$ for all $\tau >0$. Then by \eqref{1.6} we obtain (iv).
\end{proof}

  Now we are ready to prove the non-emptiness and several properties of $\textit{\textbf{N}}_3$.\par
   \begin{prop}
   Let $p=3$, $\lambda < b\lambda_1$ and $0<a<1/\Lambda$, then for every $u\in S$, \par
   {\rm (1)} $\sigma_u u \in \textit{\textbf{N}}_3$, $\tau u \in \textit{\textbf{N}}_3^{\,+}$ for $\tau \in (0, \sigma_u)$ and $\tau u \in \textit{\textbf{N}}_3^{\,-}$ for $\tau \in (\sigma_u, +\infty)${\rm ,} where $\sigma_u$ is the unique number determined by $\eqref{3.7}$, and
   \begin{eqnarray}\label{3.8}
    J(\sigma_u u) = \sup\limits_{\tau >0} J(\tau u) = \frac{\left(b\|\nabla u\|_2^2 - \lambda \|u\|_2^2\right)^2}{4 \left(\|u\|_{4}^{4} - a\|\nabla u\|_2^4\right)}{\rm ;}
    \end{eqnarray}

 {\rm (2)} $J(\tau u)$ is increasing for $\tau \in (0, \sigma_u)$, and is decreasing for $\tau \in (\sigma_u, +\infty)${\rm ;}\\
  While, if $u\in S^c$ and $\|\nabla u\|_2\neq 0$, then we have \par
 $(3)$ $I(\tau u)>0$ for all $\tau >0$.
  \end{prop}
  \begin{proof}
   Keeping in mind that $J(\tau u)= K_u(\tau)$ and $I(\tau u) = \tau K_u'(\tau)$ for $\tau>0$, then it follows from Lemma 3.1 (i), (ii) and (iv) that (1)-(3) hold except \eqref{3.8}. Further, we derive from Lemma 3.1 (iii) and \eqref{3.7} that
   \begin{eqnarray*}
    \qquad\qquad \sup\limits_{\tau >0} J(\tau u)\ =\ J(\sigma_u u) &=& -\frac{ \sigma_u^4}{4} \left(\|u\|_{4}^{4} - a\|\nabla u\|_2^4\right) + \frac{\sigma_u^2}{2} \left(b\|\nabla u\|_2^2 - \lambda \|u\|_2^2\right) \\
     &=& \frac{\left(b\|\nabla u\|_2^2 - \lambda \|u\|_2^2\right)^2}{4 \left(\|u\|_{4}^{4} - a\|\nabla u\|_2^4\right)},\qquad\qquad\qquad {\rm for}\ \ u\in S.
   \end{eqnarray*}
  Thus \eqref{3.8} is proved.
  \end{proof}

 \begin{rmk}
{\rm (1) Under the hypothesis of $p=3$, $0<a<1/\Lambda$ and $\lambda < b\lambda_1$, we have $\textit{\textbf{N}}_3$, $\textit{\textbf{N}}_3^{\,+}$, $\textit{\textbf{N}}_3^{\,-}$ are nonempty, and ${\textit{\textbf{N}}_3} \cup {\textit{\textbf{N}}_3^{\,-}} \subseteq S$.

 Indeed, noting that $\phi_\Lambda \in S$ where $\phi_\Lambda$ satisfying \eqref{1.9}, then for the number $\sigma_{\phi_\Lambda}>0$ defined by \eqref{3.7} with $u=\phi_\Lambda$, we have $\sigma_{\phi_\Lambda} \phi_\Lambda \in \textit{\textbf{N}}_3$. Moreover, we infer from Proposition 3.2 (1) that $\tau \phi_\Lambda \in \textit{\textbf{N}}_3^{\,+}$ for $0< \tau < \sigma_{\phi_\Lambda}$, and $\tau \phi_\Lambda \in \textit{\textbf{N}}_3^{\,-}$ for $\tau > \sigma_{\phi_\Lambda}$. Thus we have proved the non-emptiness of the sets $\textit{\textbf{N}}_3$, $\textit{\textbf{N}}_3^{\,+}$, $\textit{\textbf{N}}_3^{\,-}$, and we deduce from \eqref{1.18} that ${\textit{\textbf{N}}_3} \cup {\textit{\textbf{N}}_3^{\,-}} \subseteq S$.}\\[0.3em]
   {\rm (2) Keeping in mind that $\tau u \in S$ for $u\in S$ and $\tau >0$, then by Proposition 3.2 we observe that for every $u\in S$ and $\tau >0$, the ray $\tau \mapsto \tau u$  intersects exactly once the manifold $\textit{\textbf{N}}_3$ at the point $\sigma_u u$, and $\textit{\textbf{N}}_3$ separates the space $S$ into two parts:\par
      $\textit{\textbf{N}}_3^{\,+} \cap S = \left\{u\in S: I(u)>0\right\} = \left\{\tau u:~u\in S,\ \ 0<\tau<\sigma_u\right\}$, and \par
      $\textit{\textbf{N}}_3^{\,-} = \left\{u\in S: I(u)<0\right\} = \left\{\tau u:~u\in S,\ \ \tau > \sigma_u\right\} \subseteq S$.\\[0.3em]
   (3) Relation between the size of $\textit{\textbf{N}}_3^{\,+}$, $\textit{\textbf{N}}_3^{\,-}$ and the value of $a$, $\lam$.\par
   For $u\in S$ and $b>0$, by virtue of \eqref{3.7}, we find that $\sigma_u$ depends on the value of $a$ and $\lam$.\\
  $\bullet$\quad If $a\in (0, 1/\Lambda)$ is fixed, then $\sigma_u$ is strictly decreasing on $\lam \in (-\infty, b\lam_1)$. Then, for $\lam <b\lam_1$ and
 by Proposition 3.2 (1), we deduce that the size of $\textit{\textbf{N}}_3^{\,+}$ gets smaller and $\textit{\textbf{N}}_3^{\,-}$ gets bigger as $\lam$ increases.\\
  $\bullet$\quad If $\lam \in (-\infty, b\lam_1)$ is fixed, then it follows from \eqref{3.7} that  $\sigma_u$ is strictly increasing on $a \in (0, 1/\Lambda)$. Then by \eqref{3.5} and Proposition 3.2 (1) we find that $\textit{\textbf{N}}_3^{\,+}$ gets bigger, while $S$ and $\textit{\textbf{N}}_3^{\,-}$ get smaller as $a$ increases on $(0, 1/\Lambda)$.\\
  $\bullet$\quad For $a\geq 1/\Lambda$ and $\lam <b\lam_1$, then by the definition of $\Lambda$, we obtain
 $$a\|\nabla u\|_2^4 - \|u\|_4^4 \geq 0,\qquad {\rm for\ all}\ u\in H^1_0 (\Omega)\setminus \{0\}.
 $$
 Thus, we have
  $$\sup_{\tau>0} J(\tau u) = +\infty,
  $$
  and $I(u)>0$ for all $u\in H^1_0 (\Omega)\setminus \{0\}$, which indicates that $\textit{\textbf{N}}_3^{\,-}\cup \textit{\textbf{N}}_3 = \emptyset$.
      \par\smallskip
    }
    \end{rmk}

 Next, we study how the sign of $I(u)$ depends on the size of $u$.
  \begin{lem}
    Let $p=3$, $0<a<1/\Lambda$, $\lambda <b\lambda_1$, and
    $$\rho_3 := \sqrt{\frac{b_0 \Lambda}{1-a\Lambda}},
    $$
where $b_0$ appears in $\eqref{1.7}$.\par
    {\rm(1)} If $0<\|\nabla u\|_2 < \rho_3$, then $I(u)>0${\rm ;}\par
   {\rm(2)} If $I(u)<0$, then $\|\nabla u\|_2 > \rho_3${\rm ;}\par
   {\rm(3)} If $I(u)=0$, then either $\|\nabla u\|_2 =0$ or $\|\nabla u\|_2 \geq \rho_3$.
  \end{lem}
  \begin{proof}
   (1) Noting that $\lambda <b\lambda_1$, $0<a<1/\Lambda$, then using \eqref{1.6}, \eqref{1.9} and \eqref{1.12} for $p=3$, we get
   \begin{eqnarray}\label{3.9}
     I(u) &\geq& \left(a-\frac{1}{\Lambda}\right)\|\nabla u\|_2^4 + b_0 \|\nabla u\|_2^2 \nonumber\\
          &=&  \left[-\left(\frac{1}{\Lambda} -a\right) \|\nabla u\|_2^2 + b_0\right] \|\nabla u\|_2^2 .
   \end{eqnarray}
   Then $0<\|\nabla u\|_2 < \rho_3$ leads to $I(u) >0$.\par
   (2) Since $I(u)<0$, it follows from $\lambda <b\lambda_1$, $0<a<1/\Lambda$ and \eqref{3.9} that
   $$\|\nabla u\|_2^2 > \frac{b_0}{\frac{1}{\Lambda} -a} =\rho_3^2,
   $$
  Then we obtain (2).\par
   (3) If $I(u)=0$ and $\|\nabla u\|_2 \neq 0$, by the same arguments used in the proof of (2), we have $\|\nabla u\|_2 \geq \rho_3$. Thus (3) is proved.
  \end{proof}
\noindent 3.1.2  \emph{Characterization of the potential well for} $p=3$\par\medskip
   The following result concerns with the potential well depth  $d_3 = \inf\limits_{u \in \textit{\textbf{N}}_3} J(u)$, which
reveals the deep relation between $d_3$ and the minimax value of the fiber mapping. We have:
  \begin{prop}
    If $p=3$, $\lambda < b\lambda_1$ and $0<a<1/\Lambda$, then
   \begin{eqnarray}\label{3.10}
    \frac{\Lambda b_0^2}{4\left(1-a\Lambda\right)} \leq d_3 \leq  \frac{\Lambda c_1^2}{4\left(1-a\Lambda\right)}  {\rm ,}
   \end{eqnarray}
   where $b_0$ and $c_1$ appear in $\eqref{1.7}$.\par
   Moreover, we have
   \begin{eqnarray}\label{3.11}
   &&  d_3 = \inf\limits_{u \in \textit{\textbf{N}}_3} J(u) = \inf\limits_{u \in S}\sup\limits_{\tau>0} J(\tau u)\nonumber\\[0.2em]
   && \quad = \inf\limits_{u \in S} \frac{\left(b\|\nabla u\|_2^2 - \lambda \|u\|_2^2\right)^2}{4 \left(\|u\|_{4}^{4} - a\|\nabla u\|_2^4\right)}.
   \end{eqnarray}
 \end{prop}
 \begin{proof}
    For $u \in \textit{\textbf{N}}_3$, by applying \eqref{1.13} for $p=3$, and using \eqref{1.6} together with Lemma 3.4 (iii), we get
  \begin{eqnarray*}
    J(u) = \frac{b}{4} \|\nabla u\|_2^2 - \frac{\lambda}{4} \|u\|_2^2 \geq \frac{b_0}{4} \|\nabla u\|_2^2 \geq \frac{b_0}{4}   \rho_3^2,
  \end{eqnarray*}
  and $\rho_3 = \left(\frac{b_0 \Lambda}{1-a\Lambda}\right)^{1/2}$ is chosen in Lemma 3.4. Hence $d_3 \geq \frac{\Lambda b_0^2}{4\left(1-a\Lambda\right)}$.\par
    For $\phi_\Lambda$ satisfying \eqref{1.9}, noting $\phi_\Lambda \in S$ and by Proposition 3.2 (1), there exists a unique number $\sigma_{\phi_\Lambda} >0$ such that $\sigma_{\phi_\Lambda} \phi_\Lambda \in \textit{\textbf{N}}_3$. Then by \eqref{1.6} and applying \eqref{3.8} with $u=\phi_\Lambda$, we deduce that
    \begin{eqnarray*}
       J(\sigma_{\phi_\Lambda} \phi_\Lambda) = \frac{\left(b\|\nabla \phi_\Lambda\|_2^2 - \lambda \|\phi_\Lambda\|_2^2\right)^2}{4 \left(\|\phi_\Lambda\|_{4}^{4} - a\|\nabla \phi_\Lambda\|_2^4\right)}
       \leq \frac{c_1^2 \|\nabla \phi_\Lambda\|_2^4}{4 \left(\|\phi_\Lambda\|_{4}^{4} - a\|\nabla \phi_\Lambda\|_2^4\right)}
       = \frac{c_1^2 \Lambda}{4(1-a\Lambda)},
    \end{eqnarray*}
   where $c_1$ appears in \eqref{1.7}. This implies that $d_3\leq\frac{c_1^2 \Lambda}{4(1-a\Lambda)}$, and \eqref{3.10} is obtained.\par
    Due to Proposition 3.2 and in view of $\textit{\textbf{N}}_3 \subseteq S$ by Remark 3.3, then \eqref{3.11} holds.
   \end{proof}

\begin{coro}
  If $p=3$, $\lambda < b\lambda_1$ and $0<a<1/\Lambda$, then for $u\in \textit{\textbf{N}}_3^{\, -}$, we have
   \begin{eqnarray}\label{3.12}
    b\|\nabla u\|_2^2 - \lam \|u\|_2^2 > 4d_3.
   \end{eqnarray}
\end{coro}
\begin{proof}
  Since $u\in \textit{\textbf{N}}_3^{\, -} \subseteq S$ and $\lam < b\lam_1$, by \eqref{1.6} we have
  \begin{eqnarray*}
    \|u\|_4^4 - a\|\nabla u\|_2^4 > b\|\nabla u\|_2^2 - \lam \|u\|_2^2 >0.
  \end{eqnarray*}
  Therefore, \eqref{3.11} leads to
  \begin{eqnarray*}
    d_3 < \frac{1}{4}\left(b\|\nabla u\|_2^2 - \lam \|u\|_2^2\right),\qquad {\rm for}\ u\in \textit{\textbf{N}}_3^{\, -},
  \end{eqnarray*}
  and \eqref{3.12} holds.
\end{proof}

We conclude this subsection by proving the non-emptiness of the potential well $\textit{\textbf{W}}_3^{\,+} = J^{d_3} \cap \textit{\textbf{N}}_3^{\,+}$ and the unstable set $\textit{\textbf{W}}_3^{\,-} = J^{d_3} \cap \textit{\textbf{N}}_3^{\,-}$.

  To this end, we introduce the following notations. Choose
\begin{eqnarray}\label{3.13}
 R_3 =2 \sqrt{\frac{d_3}{b_0}}, \ \ {\rm and}\ \hat{r}_3 = \min\left\{\rho_3,\ \left(-\frac{c_1}{a} + \frac{1}{a}\sqrt{c_1^2 + 4ad_3}\right)^{1/2}\right\},
\end{eqnarray}
 where $\rho_3>0$ is given in Lemma 3.4, and $c_1$ appears in \eqref{1.7}. By virtue of \eqref{3.10}, we note that $R_3 \geq \rho_3$. For $r>0$, we denote by $B_r$ the set $B_r = \left\{u\in H^1_0(\Omega):\,0\leq\|\nabla u\|_2<r\right\}$, and $\bar{B}_r^{\,c} = \left\{u\in H^1_0(\Omega):\,\|\nabla u\|_2 > r\right\}$.
\begin{thm}
 If $0<a<1/\Lambda$, $\lambda < b\lambda_1$, then $\textit{\textbf{W}}_3^{\,+}\neq \emptyset$, $\textit{\textbf{W}}_3^{\,-}\neq \emptyset$ and
\begin{eqnarray}\label{3.14}
  B_{\hat{r}_3} \subseteq \textit{\textbf{W}}_3^{\,+} \subseteq B_{R_3}\cap J^{d_3},\end{eqnarray}
 \begin{eqnarray}\label{3.15}
 \bar{B}_{R_3}^{\,c}\cap J^{d_3} \subseteq \textit{\textbf{W}}_3^{\,-} \subseteq \bar{B}_{\rho_3}^{\,c}\cap J^{d_3},
\end{eqnarray}
where $R_3$, $\hat{r}_3$ are defined by $\eqref{3.13}$, and $\rho_3$ is given in Lemma $3.4$.
\end{thm}
\begin{proof}
 We begin with the proof of $\textit{\textbf{W}}_3^{\,+}\neq \emptyset$. For $u\in B_{\hat{r}_3}$, we infer from Lemma 3.4 (1) that $u \in \textit{\textbf{N}}_3^{\,+}$. Further, noting $d_3 > 0$, then it follows from \eqref{1.6}, \eqref{zh809-4} and a direct computation that
 \begin{eqnarray}
 J(u) \leq \frac{a}{4}\|\nabla u\|_2^4 + \frac{c_1}{2}\|\nabla u\|_2^2 < d_3, \ \ {\rm for}\ u\in B_{\hat{r}_3}.
 \end{eqnarray}
Thus, we have $B_{\hat{r}_3} \subseteq \textit{\textbf{W}}_3^{\,+}$ and $\textit{\textbf{W}}_3^{\,+}\neq \emptyset$. \par
 By \eqref{1.6} and \eqref{1.13} with $p=3$, we get
 \begin{eqnarray}\label{3.17}
     J(u) \geq  \frac{b_0}{4} \|\nabla u\|_2^2 + \frac{1}{p+1} I(u).
  \end{eqnarray}
  Let $u\in \textit{\textbf{W}}_3^{\,+}$ and $\|\nabla u\|_2 \neq 0$, then $I(u)>0$, and $J(u)<d_3$. So it follows from \eqref{3.17} that $\|\nabla u\|_2 <R_3$. Then we obtain \eqref{3.14}. \par
 Now we turn to prove \eqref{3.15} and $\textit{\textbf{W}}_3^{\,-}\neq \emptyset$. First, we claim that $\bar{B}_{R_3}^{\,c}\cap J^{d_3} \neq \emptyset$.\par
 Indeed, we choose $\phi_\Lambda$ satisfying \eqref{1.9}. Noting that $\phi_\Lambda \in S$ and by Proposition 3.2 (2), we deduce that there exists a unique number $\sigma_{\phi_\Lambda}>0$, such that $J(\tau \phi_\Lambda)$ is increasing for $\tau \in (0, \sigma_{\phi_\Lambda})$, and is decreasing for $\tau \in (\sigma_{\phi_\Lambda}, +\infty)$ when $p=3$, and $J(\sigma_{\phi_\Lambda} \phi_\Lambda) \geq d_3$. Further, by \eqref{3.1}, \eqref{3.6} and $p=3$, we get
\begin{eqnarray}
  J(\tau \phi_\Lambda) = -\frac{\tau^4}{4}(1-a\Lambda) + \frac{\tau^2}{2}\left(b\sqrt{\Lambda} - \lambda \|\phi_\Lambda\|_2^2\right),
\end{eqnarray}
and $\lim\limits_{\tau\rightarrow 0} J(\tau \phi_\Lambda) = 0$, $\lim\limits_{\tau\rightarrow +\infty} J(\tau \phi_\Lambda) = -\infty$ in view of $a\Lambda <1$. The above facts together with $d_3>0$ indicate that there exists $\mu_1 >0$, such that
 $\tau \phi_\Lambda\in J^{d_3}$ for every $\tau > \mu_1$. Selecting $l_1 >\max \left\{R_3\Lambda^{-\frac{1}{4}}, \mu_1\right\}$, then we have $l_1 \phi_\Lambda \in J^{d_3}\cap \bar{B}_{R_3}^{\,c}$ and $\bar{B}_{R_3}^{\,c}\cap J^{d_3} \neq \emptyset$.

 Let $u\in \bar{B}_{R_3}^{\,c}\cap J^{d_3}$, we infer from \eqref{3.13} and \eqref{3.17} that
  \begin{eqnarray}
    d_3 > J(u) \geq \frac{b_0}{4} R_3^2 + I(u) = d_3 + I(u),
  \end{eqnarray}
  which implies that $u\in \textit{\textbf{N}}_3^{\,-} \cap J^{d_3} = \textit{\textbf{W}}_3^{\,-}$. Thus, $\textit{\textbf{W}}_3^{\,-} \neq \emptyset$ and $\bar{B}_{R_3}^{\,c}\cap J^{d_3} \subseteq \textit{\textbf{W}}_3^{\,-}$.\par
  Moreover, by Lemma 3.4 (2) we have $\textit{\textbf{N}}_3^{\,-} \subseteq \bar{B}_{\rho_3}^{\,c}$, and $\textit{\textbf{W}}_3^{\,-} \subseteq \bar{B}_{\rho_3}^{\,c} \cap J^{d_3}$. Then \eqref{3.15} is proved.
 \end{proof}

\subsection{The case of $p=3$, $\lam > b\lam_1$}
3.2.1 \emph{The structure of $\textit{\textbf{N}}_3$ for $\lam > b\lam_1$}\par\medskip
   When $\lam >b\lam_1$, by virtue of $L^- \neq \emptyset$, we show that the structure of $\textit{\textbf{N}}_3$ and the behavior of $J(u)$ differ from that demonstrated in the preceding subsection.
We shall see the assumption $\|\psi_1\|_4^4 -a \|\nabla \psi_1\|_2^4 <0$ (i.e., \eqref{2.11} in Theorem 2.8) plays vital roles in establishing the properties concerning with $\textit{\textbf{N}}_3$ and $J(u)$ when $p=3$ and $\lam>b\lam_1$.\par
   The following lemma determines the value of $\delta>0$ in Theorem 2.8 and concerns the relations between the sets $L^+$, $L^-$ and $S$, where $L^+$, $L^-$ are defined in the introduction and $S$ is defined by \eqref{3.5}.
\begin{lem}
  Suppose that $\eqref{2.11}$ holds and $0<a<1/\Lambda$, then there exists $\delta >0$ such that
   \begin{eqnarray}\label{3.20}
   \overline{L^-} \cap \overline{S} =\{0\},
   \end{eqnarray}
  whenever $b\lam_1 <\lam < b\lam_1 + \delta$.
\end{lem}
\begin{proof}
  We assume that the result is false. Then for any $n\in \mathbb{N}_+$, there exist $\lam_n \in (b\lam_1, ~b\lam_1 + \frac{1}{n})$ and $u_n \in H_0^1 (\Omega)$ such that $\|\nabla u_n\|_2 =1$ and
  \begin{eqnarray}\label{3.21}
    b\|\nabla u_n\|_2^2 - \lam_n \|u_n\|_2^2 &\leq& 0,
     \end{eqnarray}
   \begin{eqnarray}\label{3.22}
    \|u_n\|_4^4 - a\|\nabla u_n\|_2^4 &\geq& 0.
  \end{eqnarray}
Noting that $\{u_n\}$ is bounded in $H_0^1(\Omega)$, we may assume that $u_n \rightharpoonup u^*$ weakly in $H_0^1(\Omega)$, and $u_n \rightarrow u^*$ strongly in $L^2(\Omega)$ and $L^4(\Omega)$, as $n\rightarrow \infty$.

 We show that $u_n \rightarrow u^*$ strongly in $H_0^1(\Omega)$. Suppose otherwise, then $\|\nabla u^*\|_2 < \varliminf\limits_{n\rightarrow \infty} \|\nabla u_n\|_2$. In view of $\lam_n \rightarrow b\lam_1$ and $u_n \rightarrow u^*$ in $L^2(\Omega)$, we deduce from \eqref{3.21} that
  $$\|\nabla u^*\|_2^2 - \lam_1 \|u^*\|_2^2 <0,
  $$
  which contradicts the definition of $\lam_1$. Thus, $u_n \rightarrow u^*$ strongly in $H_0^1(\Omega)$ and so $\|\nabla u^*\|_2 =1$.\par
  Furthermore, we deduce from \eqref{3.21} and \eqref{3.22} that
   \begin{eqnarray}
    \|\nabla u^*\|_2^2 - \lam_1 \|u^*\|_2^2 &\leq& 0, \label{3.23}\\
    \|u^*\|_4^4 - a\|\nabla u^*\|_2^4 &\geq& 0.\label{3.24}
  \end{eqnarray}
Noting \eqref{3.23} implies that $u^* = k_0 \psi_1$ for some $k_0 \in \mathbb{R}$, and then \eqref{3.24} ensures that
$$k_0^4 \left( \|\psi_1\|_4^4 - a\|\nabla \psi_1\|_2^4\right) \geq 0.
$$
Then, it follows from the assumption \eqref{2.11} that $k_0 =0$, which is impossible as $\|\nabla u^*\|_2 =1$. Hence we obtain \eqref{3.20}.
\end{proof}

  By Lemma 3.8, we infer that there exists $\delta>0$ such that $S\subseteq L^+$ and $L^- \subseteq {\overline{S}}^{\,c}$ for any $\lam \in (b\lam_1,\, b\lam_1 +\delta)$. Basing on this fact, we prove the following properties of $\textit{\textbf{N}}_3$ with $p=3$ and $\lam>b\lam_1$, which differ from the case $\lam <b\lam_1$ showed in Proposition 3.2.
\begin{prop}
  Suppose that \eqref{2.11} holds, $0<a<1/\Lambda$ and $b\lam_1 <\lam <b\lam_1 +\delta$, where $\delta>0$ is defined by Lemma $3.8$. Let $\sigma_u$ be the unique number determined by $\eqref{3.7}$, then$:$ \par
  {\rm(i)} $\textit{\textbf{N}}_3 \cap L^0 = \emptyset$ and $0\not\in \overline{\textit{\textbf{N}}_3 \cap L^+};$\par
  {\rm(ii)} $\textit{\textbf{N}}_3 \cap L^+$ is nonempty and closed$;$ Moreover, for every $u\in S$, we have $\sigma_u u \in \textit{\textbf{N}}_3 \cap L^+$, $\tau u \in \textit{\textbf{N}}_3^{\,+}$ for $\tau \in (0, \sigma_u)$ and $\tau u \in \textit{\textbf{N}}_3^{\,-}$ for $\tau \in (\sigma_u, +\infty)${\rm ;}\par
  {\rm(iii)} $\textit{\textbf{N}}_3 \cap L^-$ is nonempty and bounded$;$ Moreover, for every $u\in L^-$, we have $\sigma_u u \in \textit{\textbf{N}}_3 \cap L^-$, $\tau u \in \textit{\textbf{N}}_3^{\,-}$ for $\tau \in (0, \sigma_u)$ and $\tau u \in \textit{\textbf{N}}_3^{\,+}$ for $\tau \in (\sigma_u, +\infty)${\rm .}
 \end{prop}
\begin{proof}
(i) To show $\textit{\textbf{N}}_3 \cap L^0 = \emptyset$, we suppose by contradiction that $\textit{\textbf{N}}_3 \cap L^0 \neq \emptyset$ and $v \in \textit{\textbf{N}}_3 \cap L^0$.
Then we have
$$a\|\nabla v\|_2^4 - \|v\|_4^4 = \lam \|v\|_2^2 -b\|\nabla v\|_2^2 =0,
$$
and $v\in \overline{L^-} \cap \overline{S}$. Hence \eqref{3.20} implies that $\|\nabla v\|_2 =0$, which contradicts $v\in \textit{\textbf{N}}_3$.\par
Next, we prove $0\not\in \overline{\textit{\textbf{N}}_3 \cap L^+}$. Suppose otherwise; then there exists $u_n \in \textit{\textbf{N}}_3 \cap L^+$ such that $u_n \rightarrow 0$ in $H_0^1(\Omega)$. Then
\begin{eqnarray}\label{3.25}
  0< b\|\nabla u_n\|_2^2 - \lam \|u_n\|_2^2 = \|u_n\|_4^4 -a \|\nabla u_n\|_2^4 \rightarrow 0,\ \ {\rm as}\ n\rightarrow \infty.
\end{eqnarray}

   To get a contradiction, we let
   $$v_n = \frac{u_n}{\|\nabla u_n\|_2}.
   $$
    Since $\|\nabla v_n\|_2 =1$, we may assume that $v_n \rightharpoonup v_0$ weakly in $H_0^1(\Omega)$, $v_n \rightarrow v_0$ strongly in $L^2(\Omega)$ and $L^4(\Omega)$, as $n\rightarrow \infty$. Dividing \eqref{3.25} by $\|\nabla u_n\|_2^2$ gives
   \begin{eqnarray}\label{3.26}
     0 &<& b\|\nabla v_n\|_2^2 - \lam \|v_n\|_2^2 \nonumber \\
     &=& \frac{\|u_n\|_4^4 - a\|\nabla u_n\|_2^4}{\|\nabla u_n\|_2^2} \nonumber\\
     &=& \left(\|v_n\|_4^4 -a\right)\|\nabla u_n\|_2^2 \rightarrow 0,\ \ {\rm as}\ n\, \rightarrow\, \infty.
   \end{eqnarray}
  Thus, it follows from \eqref{3.26} and the fact that $v_n \rightarrow v_0$ in $L^2(\Omega)$ that
  $$0= \lim\limits_{n\rightarrow \infty} \left(b\|\nabla v_n\|_2^2 - \lam \|v_n\|_2^2\right) = b-\lam \|v_0\|_2^2,
  $$
  which implies that $\|\nabla v_0\|_2 \neq 0$.\par
    On the other hand, we deduce from \eqref{3.26} that
    $$b\|\nabla v_0\|_2^2 - \lam \|v_0\|_2^2 \leq \varliminf\limits_{n\rightarrow \infty} \left(b\|\nabla v_n\|_2^2 - \lam \|v_n\|_2^2\right) =0,
    $$
 which implies that $v_0 \in \overline{L^-}$.\par
   Moreover, by virtue of \eqref{3.26}, we find that $\|v_n\|_4^4 >a$ and $\|v_0\|_4^4 \geq a$. Observing that
 $\|\nabla v_0\|_2^2 \leq \varliminf\limits_{n\rightarrow \infty} \|\nabla v_n\|_2^2 =1$, hence we have
 $\|v_0\|_4^4 - a\|\nabla v_0\|_2^4 \geq 0$ and $v_0 \in \overline{L^-} \cap \overline{S}$. Then \eqref{3.20} implies that $\|\nabla v_0\|_2 =0$, which leads to a contradiction. Therefore $0\not\in \overline{\textit{\textbf{N}}_3 \cap L^+}$.

 (ii) For $u\in S$, we deduce from Lemma 3.8 that $u\in L^+$ and $\|\nabla u\|_2 \neq 0$. Using the same argument as in the proof of Proposition 3.2 (1), we obtain $h_3(\sigma_u)= 0$, $h_3(\tau)> 0$ for $\tau \in (0, \sigma_u)$ and $h_3(\tau)< 0$ for $\tau \in (\sigma_u, +\infty)$, where $h_3(\tau)$ is the function defined by \eqref{3.4}. Then it follows that $\sigma_u u \in \textit{\textbf{N}}_3 \cap L^+$ and $\textit{\textbf{N}}_3 \cap L^+ \neq \emptyset$. Moreover, we have $\tau u \in \textit{\textbf{N}}_3^{\,+}$ for $\tau \in (0, \sigma_u)$ and $\tau u \in \textit{\textbf{N}}_3^{\,-}$ for $\tau \in (\sigma_u, +\infty)$.\par
 Next, we show $\textit{\textbf{N}}_3 \cap L^+$ is closed. By (i) we get $\overline{\textit{\textbf{N}}_3 \cap L^+} \subseteq \left(\textit{\textbf{N}}_3 \cap L^+\right) \cup \{0\}$. Then in view of $0\not\in \overline{\textit{\textbf{N}}_3 \cap L^+}$, we have
$\overline{\textit{\textbf{N}}_3 \cap L^+} = \textit{\textbf{N}}_3 \cap L^+$ and prove (ii).\par
  (iii) For $u\in L^-$, noting that $L^- \subseteq {\overline{S}}^{\,c}$, we have $b\|\nabla u\|_2^2 - \lam \|u\|_2^2 <0$ and
$\|u\|_4^4 - a \|\nabla u\|_2^4 <0$. Thus the number $\sigma_u >0$ which appears in \eqref{3.7} is also well defined as $u\in L^-$ and $\lam \in (b\lam_1,\, b\lam_1 +\delta)$. By a direct computation, we deduce from \eqref{3.4} that
  $$h_3(\tau) = \left(a \|\nabla u\|_2^4 - \|u\|_4^4 \right) \left(\tau^2 -\sigma_u ^2\right).
  $$
Therefore $\sigma_u u \in \textit{\textbf{N}}_3 \cap L^-$ and $\textit{\textbf{N}}_3 \cap L^- \neq \emptyset$. Furthermore, noting that $a \|\nabla u\|_2^4 - \|u\|_4^4 >0$, then we have $\tau u \in \textit{\textbf{N}}_3^{\,-}$ for $\tau \in (0, \sigma_u)$ and $\tau u \in \textit{\textbf{N}}_3^{\,+}$ for $\tau \in (\sigma_u, +\infty)$.\par
   Next, we prove that $\textit{\textbf{N}}_3 \cap L^-$ is bounded via a contradiction argument. Suppose that
$\textit{\textbf{N}}_3 \cap L^-$ is unbounded, then there exists a sequence $\{u_n\}\subseteq \textit{\textbf{N}}_3 \cap L^-$ such that $\|\nabla u_n\|_2 \rightarrow \infty$ as $n\rightarrow \infty$, and
\begin{eqnarray}\label{3.27}
  b\|\nabla u_n\|_2^2 -\lam \|u_n\|_2^2 = \|u_n\|_4^4 - a \|\nabla u_n\|_2^4 <0.
\end{eqnarray}

Putting
$$v_n = \frac{u_n}{\|\nabla u_n\|_2},
   $$
 then $\|\nabla v_n\|_2 =1$ and $v_n \in L^-$. We may assume that $v_n \rightharpoonup v_0$
weakly in $H_0^1(\Omega)$, $v_n \rightarrow v_0$ strongly in $L^2(\Omega)$ and $L^4(\Omega)$ as $n\rightarrow \infty$.
  We divide the proof into the following two steps.

  {\emph {Step}} 1: We show that $\|\nabla v_0\|_2 >0$ and $v_0 \in \overline{S}$.\par
   From \eqref{3.27}, we infer that
   \begin{eqnarray}\label{3.28}
     b- \lam \|v_n\|_2^2 = \left(\|v_n\|_4^4 -a\right)\|\nabla u_n\|_2^2.
   \end{eqnarray}
In view of \eqref{1.6} and $v_n \in L^-$, we get
 $$b-\frac{\lam}{\lam_1} \leq b- \lam \|v_n\|_2^2 \leq 0.
 $$
Moreover, noting that $\|\nabla u_n\|_2 \rightarrow \infty$, then we deduce from \eqref{3.28} that $\|v_n\|_4^4 \rightarrow a$ as $n\rightarrow \infty$, and so $\|v_0\|_4^4 =a$. It follows from $\|\nabla v_0\|_2 \leq \varliminf\limits_{n\rightarrow \infty} \|\nabla v_n\|_2 =1$ that $\|v_0\|_4^4 - a\|\nabla v_0\|_2^4 \geq 0$. Thus, by virtue of $\|v_0\|_4^4 =a>0$, we obtain $\|\nabla v_0\|_2 >0$ and $v_0 \in \overline{S}$.\par
    {\emph {Step}} 2: We show that $v_n \rightarrow v_0$ in $H_0^1(\Omega)$ as $n\rightarrow \infty$.\par
    Suppose that this is false. Then $\|\nabla v_0\|_2 < \varliminf\limits_{n\rightarrow \infty} \|\nabla v_n\|_2$. By virtue of $v_n \in L^-$, we have
    \begin{eqnarray}\label{3.29}
      b\|\nabla v_0\|_2^2 - \lam \|v_0\|_2^2 < \varliminf\limits_{n\rightarrow \infty} \left(b\|\nabla v_n\|_2^2 - \lam \|v_n\|_2^2\right)\leq 0.
    \end{eqnarray}
Combining \eqref{3.29} with the result obtained in Step 1, we have $v_0 \in \overline{L^-} \cap \overline{S}$ and $\|\nabla v_0\|_2 >0$, which contradicts \eqref{3.20}.\par
  Hence, it follows from $v_n \in L^-$ and $v_n \rightarrow v_0$ in $H_0^1(\Omega)$ that $b\|\nabla v_0\|_2^2 - \lam \|v_0\|_2^2 \leq 0$ and $v_0 \in \overline{L^-}$, which is impossible. Therefore, we obtain (iii).
\end{proof}
\noindent 3.2.2 \emph{Properties of $J(u)$ on $\textit{\textbf{N}}_3$ for $\lam > b\lam_1$}\par\medskip
   The following lemma concerns with the properties of $J(u)$ on $\textit{\textbf{N}}_3 \cap L^+$ and $\textit{\textbf{N}}_3 \cap L^-$.
\begin{lem}
  Under the conditions of Proposition $3.9$, we have{\rm :}\\[0.2em]
{\rm (1)} $J(u) >0$ for any $u\in \textit{\textbf{N}}_3 \cap L^+${\rm ;}\\
{\rm (2)} every minimizing sequence of $J(u)$ on $\textit{\textbf{N}}_3 \cap L^+$ is bounded in $H_0^1(\Omega)${\rm;}\\
{\rm (3)} there exists $C_1 >0$ such that $-C_1 \leq J(u) <0$ for any $u\in \textit{\textbf{N}}_3 \cap L^-$.
\end{lem}
 \begin{proof}
  (1) For any $u \in \textit{\textbf{N}}_3 \cap L^+$, we have $\|u\|_4^4 - a \|\nabla u\|_2^4 = b\|\nabla u\|_2^2 -\lam \|u\|_2^2 >0$. Thus by \eqref{zh809-4} with $p=3$, we obtain that
  $$J(u) = \frac{1}{4}\left(b\|\nabla u\|_2^2 -\lam \|u\|_2^2\right) >0,\ \ \ \ {\rm for\ all}\ u \in \textit{\textbf{N}}_3 \cap L^+.
  $$
 (2) Let $\{u_n\}$ be a minimizing sequence of $J(u)$ on $\textit{\textbf{N}}_3 \cap L^+$, then there exists $c\geq 0$ such that
 \begin{eqnarray}\label{3.30}
    b\|\nabla u_n\|_2^2 -\lam \|u_n\|_2^2 = \|u_n\|_4^4 - a \|\nabla u_n\|_2^4 \rightarrow c,\ \ \ {\rm as}\ n\rightarrow \infty.
 \end{eqnarray}

  By a contradiction we assume that $\|\nabla u_n\|_2^2 \rightarrow \infty$ as $n\rightarrow \infty$. Let
   $$v_n = \frac{u_n}{\|\nabla u_n\|_2},
   $$
   thus we may assume that $v_n \rightharpoonup v_0$ weakly in $H_0^1(\Omega)$, $v_n \rightarrow v_0$ strongly in $L^2(\Omega)$ and $L^4(\Omega)$ as $n\rightarrow \infty$. Then it follows from \eqref{3.30} and $\|\nabla u_n\|_2 \rightarrow \infty$ that
 \begin{eqnarray}\label{3.31}
   b-\lam\|v_n\|_2^2 = \left(\|v_n\|_4^4 -a\right)\|\nabla u_n\|_2^2 \rightarrow 0,\ \ \ {\rm as}\ n\rightarrow \infty,
 \end{eqnarray}
 and so $\|v_n\|_4^4 \rightarrow a$ as $n\rightarrow \infty$. Since $v_n \rightarrow v_0$ in $L^4(\Omega)$, we infer that $\|v_0\|_4^4 =a$. Furthermore, we deduce from $\|\nabla v_0\|_2 \leq \varliminf\limits_{n\rightarrow \infty} \|\nabla v_n\|_2 =1$ that $\|v_0\|_4^4 - a\|\nabla v_0\|_2^4 \geq 0$. Thus $v_0 \in \overline{S}$.\par
   On the other hand, by virtue of $v_n \rightharpoonup v_0$ weakly in $H_0^1(\Omega)$ and \eqref{3.31}, we have
   \begin{eqnarray*}
      b\|\nabla v_0\|_2^2 - \lam \|v_0\|_2^2 \leq \varliminf\limits_{n\rightarrow \infty} \left(b\|\nabla v_n\|_2^2 - \lam \|v_n\|_2^2\right)= 0,
    \end{eqnarray*}
 which implies that $v_0 \in \overline{L^-}$. Then by \eqref{3.20} we get $\|\nabla v_0\|_2 =0$, which is impossible as $\|v_0\|_4^4 =a>0$. Therefore, $\{u_n\}$ is bounded.\\[0.2em]
 (3) For any $u \in \textit{\textbf{N}}_3 \cap L^-$, we deduce from \eqref{zh809-4} for $p=3$ and $\|u\|_4^4 - a \|\nabla u\|_2^4 = b\|\nabla u\|_2^2 -\lam \|u\|_2^2 <0$ that
  $$J(u) = \frac{1}{4}\left(\|u\|_4^4 - a \|\nabla u\|_2^4\right) <0.
  $$

  Moreover, by virtue of Proposition 3.9 (iii), we find that there exists $C>0$ such that $\|\nabla u\|_2 \leq C$ for all $u \in \textit{\textbf{N}}_3 \cap L^-$, and
  $$J(u) \geq -\frac{a}{4}\|\nabla u\|_2^4 \geq -\frac{a}{4} C^4.
  $$
  Hence we prove (3).
 \end{proof}

    From Lemma 3.10, we can see that $d_3^{\,+} = \inf\limits_{u\in \textit{\textbf{N}}_3 \cap L^+} J(u)$ and
 $d_3^{\,-} = \inf\limits_{u\in \textit{\textbf{N}}_3 \cap L^-} J(u)$ are well defined.\par\medskip
  At the end of this subsection, we give the following lemma to show the relation between $d_3^{\,-}$ and $\textit{\textbf{N}}_3^{\,-}$ for $\lam >b\lam_1$.
 \begin{lem}
  Under the conditions of Proposition $3.9$, we have{\rm :}
  {\rm(1)} $d_3^{\,-} <0< d_3^{\,+}${\rm;} {\rm(2)} $\textit{\textbf{N}}_3^{\,-} \cap J^{d_3^{\,-}} \neq \emptyset${\rm;}
  {\rm(3)} For any $u\in \textit{\textbf{N}}_3^{\,-},$
  \begin{eqnarray}\label{3.32}
    d_3^{\,-}< \frac{1}{4}\left(b\|\nabla u\|_2^2 -\lam \|u\|_2^2\right).
  \end{eqnarray}
\end{lem}
 \begin{proof} (1) It follows from Lemma 3.10 (3) that $d_3^{\,-}<0$, and it suffices for us to prove that $d_3^{\,+}>0$.\par If this is false, then by virtue of Lemma 3.10 (1), we have $\inf\limits_{u\in \textit{\textbf{N}}_3 \cap L^+} J(u) = 0$. Let $\{u_n\}$ be a minimizing sequence of $J(u)$ on $\textit{\textbf{N}}_3 \cap L^+$, then
 \begin{eqnarray}\label{3.33}
   b\|\nabla u_n\|_2^2 -\lam \|u_n\|_2^2 = \|u_n\|_4^4 - a \|\nabla u_n\|_2^4 \rightarrow 0 \qquad {\rm as}\ n\rightarrow \infty.
 \end{eqnarray}
 In view of Lemma 3.10 (2), we may assume that $u_n \rightharpoonup \hat{u}$
weakly in $H_0^1(\Omega)$, $u_n \rightarrow \hat{u}$ strongly in $L^2(\Omega)$ and $L^4(\Omega)$. Furthermore, we can prove that $u_n \rightarrow \hat{u}$ in $H_0^1(\Omega)$ by the similar manner as used in the Step 2 of Proposition 3.9 (iii). Then, it follows from Proposition 3.9 (ii) that $\|\nabla \hat{u}\|_2 \neq 0$.\par
    On the other hand, we deduce from \eqref{3.33} and $u_n \rightarrow \hat{u}$ in $H_0^1(\Omega)$ that
    \begin{eqnarray*}
   b\|\nabla \hat{u}\|_2^2 -\lam \|\hat{u}\|_2^2 = \|\hat{u}\|_4^4 - a \|\nabla \hat{u}\|_2^4 =0.
 \end{eqnarray*}
Then by virtue of \eqref{3.20}, we have $\hat{u} \in \overline{L^{\,-}} \cap \overline{S} = \{0\}$ which contradicts $\|\nabla \hat{u}\|_2 \neq 0$, and we obtain (1).\par
(2) Let $u\in S$, then it follows from \eqref{3.20} that $\|u\|_4^4 - a \|\nabla u\|_2^4 >0$ and $b\|\nabla u\|_2^2 -\lam \|u\|_2^2 >0$. Now we choose $k>0$ such that $ku \in \textit{\textbf{N}}_3^{\,-} \cap J^{d_3^{\,-}}$. \par
   By \eqref{1.12} with $p=3$, we get
   $$I(ku) = -k^2 \left(\|u\|_4^4 - a \|\nabla u\|_2^4\right) (k^2 - \sigma_u^2),
   $$
 where $\sigma_u >0$ is the number determined by \eqref{3.7}. Hence we have $I(ku) <0$ for $k>\sigma_u$.\par
    Furthermore, we deduce from \eqref{zh809-4} that
    $$J(ku) = -\frac{k^2}{4} \left(\|u\|_4^4 - a \|\nabla u\|_2^4\right) (k^2 - 2\sigma_u^2).
    $$
 Thus, there exists $K_1 >\sqrt{2}\sigma_u$ such that $J(ku) < d_3^{\, -}$ for $k>K_1$, which leads to
 $ku \in \textit{\textbf{N}}_3^{\,-} \cap J^{d_3^{\,-}}$ for $k>K_1$. Hence we prove (2).\par
    (3) If $u\in \textit{\textbf{N}}_3^{\,-} \cap \overline{L^+}$, then we deduce from (1) that
 $b\|\nabla u\|_2^2 -\lam \|u\|_2^2 \geq 0 >4d_3^{\, -}$.\par
   If $u\in \textit{\textbf{N}}_3^{\,-} \cap L^-$, noting $\|\nabla u\|_2 \neq 0$ and \eqref{3.20}, then we have $u\in {\overline{S}}^{\,c}$. Hence, it follows from $I(u)<0$ that
   \begin{eqnarray}\label{3.34}
     b\|\nabla u\|_2^2 -\lam \|u\|_2^2 < \|u\|_4^4 - a \|\nabla u\|_2^4 <0.
   \end{eqnarray}
On the other hand, in view of Proposition 3.9 (iii), we know that for any $u\in L^-$, there exists a unique number
 $\sigma_u >0$ such that $\sigma_u u \in \textit{\textbf{N}}_3 \cap L^-$, and so
 \begin{eqnarray}\label{3.35}
   d_3^{\,-} \leq J(\sigma_u u) = \frac{\left(b\|\nabla u\|_2^2 -\lam \|u\|_2^2\right)^2}{4\left(\|u\|_4^4 - a \|\nabla u\|_2^4\right)}.
 \end{eqnarray}
 Combining \eqref{3.34} and \eqref{3.35}, we have
$4d_3^{\,-}< b\|\nabla u\|_2^2 -\lam \|u\|_2^2$ for any $u\in \textit{\textbf{N}}_3^{\,-} \cap L^-$. Thus, for all $u\in \textit{\textbf{N}}_3^{\,-}$ we obtain \eqref{3.32}.
 \end{proof}

\section{Structure of Nehari manifold $\textit{\textbf{N}}_p$ for $3<p<5$, $\lam \leq b\lam_1$}
\setcounter{equation}{0}
\label{}\ \ \ \ \ This section is devoted to show the structure of Nehari manifold $\textit{\textbf{N}}_p$ for $3<p<5$.
\subsection{The structure of $\textit{\textbf{N}}_p$ for $3<p<5$ and $\lam \leq b\lam_1$}
 To begin with, we study the behavior of fiber mapping $K_u : \mathbb{R}^+ \rightarrow \mathbb{R}$ defined by \eqref{3.1} for $3<p<5$.
\begin{lem}
 Let $3<p<5$ and $\lambda \leq b\lambda_1$, then for every $u\in H^1_0(\Omega)\backslash \{0\}$, there holds{\rm :}\par
 {\rm (i)} there is a unique number $\tau_u>0$ depending on $u$, such that $K_u'\left(\tau_u\right) = 0${\rm ;}\par
 {\rm (ii)} $K_u'\left(\tau\right) > 0$ for $0< \tau< \tau_u${\rm,} and $K_u'\left(\tau\right) < 0$ for $\tau > \tau_u${\rm ;}\par
 {\rm (iii)} $K_u''\left(\tau_u\right) < 0$, and $K_u(\tau_u) =  \sup\limits_{\tau >0}K_u(\tau)$.
\end{lem}
\begin{proof}
 (i)  For $3<p<5$ and every $u\in H^1_0(\Omega)\backslash \{0\}$, let
  \begin{eqnarray*}
    h_p(\tau)= a\tau^2 \|\nabla u\|_2^4 + b \|\nabla u\|_2^2 - \lambda \|u\|_2^2 - \tau^{p-1} \|u\|_{p+1}^{p+1}.
  \end{eqnarray*}
     From \eqref{3.2} we have $K_u'(\tau)=\tau h_p (\tau)$. We will prove that there is a unique $\tau_u >0$ such that $h_p (\tau_u)=0$.\par
     Since $3<p<5$, $\|\nabla u\|_2 \neq 0$ and
  \begin{eqnarray*}
    h'_p(\tau) = 2a\tau \|\nabla u\|_2^4 - (p-1)\tau^{p-2} \|u\|_{p+1}^{p+1},
  \end{eqnarray*}
 we infer that $h_p(\tau)$ admits a unique stationary point
 \begin{eqnarray}\label{4.1}
   \tau_0 = \left[\frac{2a\|\nabla u\|_2^4}{(p-1)\|u\|_{p+1}^{p+1}}\right]^{\frac{1}{p-3}} >0,
 \end{eqnarray}
 and $h''_p(\tau_0) = 2a \|\nabla u\|_2^4 - (p-1)(p-2)\tau_0^{p-3} \|u\|_{p+1}^{p+1} = 2a(3-p)\|\nabla u\|_2^4  < 0$.
 Thus $h_p(\tau)$ is strictly increasing for $\tau\in [0,\tau_0)$ and strictly decreasing for $\tau\in (\tau_0, +\infty)$. We can check that $h_p (\tau)$ attains its maximum at $\tau_0$ with
  \begin{eqnarray*}
    h_p(\tau_0) &=& \tau_0^2 \left [ a\|\nabla u\|_2^4 - \tau_0^{p-3} \|u\|_{p+1}^{p+1} \right] + b \|\nabla u\|_2^2 - \lambda \|u\|_2^2 \\
    &=& a\tau_0^2 \|\nabla u\|_2^4 \left ( 1 - \frac{2}{p-1} \right) + b \|\nabla u\|_2^2 - \lambda \|u\|_2^2 \\
    &\geq& \frac{p-3}{p-1} a\tau_0^2 \|\nabla u\|_2^4 + b_0 \|\nabla u\|_2^2 \ >\ 0,
  \end{eqnarray*}
  by virtue of \eqref{1.6} and $3<p<5$.\par
  Moreover, it follows from \eqref{1.6} and $p>3$ that $h_p(0) = b \|\nabla u\|_2^2 - \lambda \|u\|_2^2 \geq b_0 \|\nabla u\|_2^2 \geq 0$, and $h_p(\tau) \rightarrow -\infty$ as $\tau \rightarrow +\infty$. Hence, we show that there is a unique $\tau_u >0$ such that $h_p (\tau_u)=0$. Then (i) holds. Furthermore, we have $\tau_u >\tau_0$, where $\tau_0$ is determined by \eqref{4.1}.\par
    (ii) By (i) we have $h_p (\tau) > 0$ for $0< \tau< \tau_u$, and $h_p (\tau) < 0$ for $\tau > \tau_u$. Thus, keeping in mind that $K_u'(\tau)=\tau h_p (\tau)$ we obtain (ii). \par
    (iii) From (i), $K_u'(\tau_u)=0$, then
   \begin{eqnarray}\label{4.2}
    \tau_u^{p-1} \|u\|_{p+1}^{p+1} = a\tau_u^2 \|\nabla u\|_2^4 + b \|\nabla u\|_2^2 - \lambda \|u\|_2^2.
   \end{eqnarray}
  Combining with \eqref{3.3}, \eqref{4.2} and \eqref{1.6}, we have
  \begin{eqnarray*}
   K''_u(\tau_u) &=& -(p-3)\,a\, \tau_u^2 \|\nabla u\|_2^4 -(p-1)\,b\, \|\nabla u\|_2^2 + (p-1)\,\lambda\, \|u\|_2^2\\
     &\leq&
      -(p-3)\,a\, \tau_u^2\, \|\nabla u\|_2^4 -(p-1)b_0 \|\nabla u\|_2^2 \ <\ 0,
  \end{eqnarray*}
and $K_u(\tau_u) =  \sup\limits_{\tau >0}K_u(\tau)$.
\end{proof}

  Basing on Lemma 4.1, we have
  \begin{prop}
    Let $3<p<5$,  $\lambda \leq b\lambda_1$, and $\tau_u >0$ be the unique constant given in {\rm Lemma} $4.1$, then for every $u\in H^1_0(\Omega)\backslash \{0\}$, there holds{\rm :}\par
 {\rm (1)} $\tau_u u \in \textit{\textbf{N}}_p$ and $J(\tau_u u) = \sup\limits_{\tau >0} J(\tau u)$. Moreover, $\tau u \in \textit{\textbf{N}}_p^{\,+}$ for $\tau \in (0, \tau_u)$ and $\tau u \in \textit{\textbf{N}}_p^{\,-}$ for $\tau \in (\tau_u, +\infty)${\rm ;}

 {\rm (2)} $J(\tau u)$ is increasing for $\tau \in (0, \tau_u)$, and is decreasing for $\tau \in (\tau_u, +\infty);$

 {\rm (3)} $\tau_u$ is decreasing with respect to $\lam$ on $(-\infty, b\lam_1]$.
  \end{prop}
  \begin{proof}
 The claim (1) can be obtained by Lemma 4.1 (i), (iii) directly.\par
   Observing that $J(\tau u)= K_u(\tau)$ and $I(\tau u) = \tau K_u'(\tau)$ for $\tau>0$, then we have (2) by virtue of Lemma 4.1 (ii).\par
   To prove (3), from \eqref{4.2} we get a function $\lam=\lam(\tau_u)$ which is defined by
   \begin{eqnarray}\label{4.3}
     \lam = \frac{- \tau_u^{p-1} \|u\|_{p+1}^{p+1} + a\tau_u^2 \|\nabla u\|_2^4 + b \|\nabla u\|_2^2}{\|u\|_2^2},
   \end{eqnarray}
    for $u\in H^1_0(\Omega)\setminus \{0\}$ is fixed.
  Noting that $\tau_u >\tau_0$, where $\tau_0$ is determined by \eqref{4.1}, then we have
   \begin{eqnarray*}
     \frac{\dif \lam}{\dif \tau_u} &=&  \frac{-(p-1)\tau_u^{p-2} \|u\|_{p+1}^{p+1} + 2a\tau_u \|\nabla u\|_2^4}{ \|u\|_2^2}\\
     &=& -\frac{\tau_u}{\|u\|_2^2}\left[(p-1)\tau_u^{p-3} \|u\|_{p+1}^{p+1} - 2a \|\nabla u\|_2^4\right] <0.
   \end{eqnarray*}
  Hence $\lam$ is decreasing with respect to $\tau_u$. Observing that $\tau_u$ is the inverse function of \eqref{4.3}, then we obtain (3).
  \end{proof}
\begin{rmk}
  {\rm Proposition 4.2 indicates several important facts of $\textit{\textbf{N}}_p$, $\textit{\textbf{N}}_p^{\,+}$ and $\textit{\textbf{N}}_p^{\,-}$ for $3<p<5$. \par
  (a) $\textit{\textbf{N}}_p$, $\textit{\textbf{N}}_p^{\,+}$, $\textit{\textbf{N}}_p^{\,-} \neq \emptyset$ for $3<p<5$ and $\lambda \leq b\lambda_1$; \par
  (b) For every $u\in H^1_0(\Omega)\backslash \{0\}$ and $\tau >0$, the ray $\tau \mapsto \tau u$ intersects $\textit{\textbf{N}}_p$ at the unique point $\tau_u u$. From the viewpoint of geometry, each ray starting from the origin of $H^1_0(\Omega)$ intersects exactly once the manifold $\textit{\textbf{N}}_p$, and $H^1_0(\Omega)$ is separated by $\textit{\textbf{N}}_p$ into two parts: the inside part is
  $$\textit{\textbf{N}}_p^{\,+} = \left\{ u \in H^1_0(\Omega): I(u)>0\right\} \cup \{0\}
       = \left\{\tau u: u\in H^1_0(\Omega),\ 0\leq \tau<\tau_u\right\},
   $$
   and the outside part is
   $$\textit{\textbf{N}}_p^{\,-} = \left\{ u \in H^1_0(\Omega): I(u)<0\right\}
       = \left\{\tau u: u\in H^1_0(\Omega),\  \tau > \tau_u\right\}.\qquad\qquad\quad
   $$

   (c)  For $\lam \leq b\lam_1$, we know from Proposition 4.2 (1) and (3) that the size of $\textit{\textbf{N}}_p^{\,+}$ gets smaller and $\textit{\textbf{N}}_p^{\,-}$ gets bigger as $\lam$ increases.
  }\hfill $\Box$
\end{rmk}

\begin{lem}
   Let $3<p<5$, $\lambda \leq b\lambda_1$, and $\rho_p = \left(S_{p+1}^{(p+1)/2}a\right)^{\frac{1}{p-3}}$, where $S_{p+1}$ is the Sobolev constant given by $\eqref{1.8}$ with exponent $q=p+1$.\par
  {\rm(1)} If $0< \|\nabla u\|_2 <\rho_p$, then $I(u)>0${\rm ;}\par
   {\rm(2)} If $I(u)<0$, then $\|\nabla u\|_2 > \rho_p${\rm ;}\par
   {\rm(3)} If $I(u)=0$, then either $\|\nabla u\|_2 =0$ or $\|\nabla u\|_2 \geq \rho_p$.
\end{lem}
\begin{proof}
 (1) If $0< \|\nabla u\|_2 <\rho_p$, then by \eqref{1.8},
   $$\|u\|_{p+1}^{p+1} \leq \frac{1}{S_{p+1}^{(p+1)/2}} \|\nabla u\|_2^{p+1} < \frac{\rho_p^{\,p-3}}{S_{p+1}^{(p+1)/2}}  \|\nabla u\|_2^4 = a  \|\nabla u\|_2^4.
   $$
   Combining the above inequality with \eqref{1.7} and \eqref{1.12}, we have $I(u)>0$.\par
   (2) By contradiction suppose that $\|\nabla u\|_2 \leq \rho_p$, then we have $I(u)\geq 0$ with a similar argument in (1), which contradicts with the assumption $I(u)<0$. \par
   (3) If $I(u)=0$ and $\|\nabla u\|_2 \neq 0$, it follows from \eqref{1.6}, \eqref{1.8} and \eqref{1.12} that
   \begin{eqnarray*}
   a\|\nabla u\|_4^4 &=& - b \|\nabla u\|_2^2 + \lambda \|u\|_2^2 + \|u\|_{p+1}^{p+1}\\
    &\leq& \|u\|_{p+1}^{p+1}  \ \leq\ \frac{1}{S_{p+1}^{(p+1)/2}} \|\nabla u\|_2^{p+1}.
   \end{eqnarray*}
   Then we deduce that  $\|\nabla u\|_2 \geq \left(S_{p+1}^{(p+1)/2}a\right)^{\frac{1}{p-3}}= \rho_p$ by virtue of $p>3$.
\end{proof}

\subsection{The structure of $\textit{\textbf{W}}_p^{\,+}$ and $\textit{\textbf{W}}_p^{\,-}$ for $3<p<5$, $\lam \leq b\lam_1$}
  First, we show the positivity and minimax characterization of $d_p$ for $3<p<5$ defined by \eqref{1.16}.
 \begin{prop}
    If $3<p<5$ and $ \lambda \leq b\lambda_1$, then $0< d_p <+\infty${\rm ,} and
   \begin{eqnarray}\label{4.4}
   d_p = \inf\limits_{u \in \textit{\textbf{N}}_p} J(u) = \inf\limits_{u \in H^1_0(\Omega)\backslash \{0\}}\sup\limits_{\tau>0} J(\tau u).
   \end{eqnarray}
  \end{prop}
 \begin{proof}
   At first, we prove $d_p >0$. If $u \in \textit{\textbf{N}}_p$ for $3<p<5$, we deduce from \eqref{1.6} and \eqref{1.13} that
  \begin{eqnarray}\label{4.5}
     J(u) &=& a\left(\frac{1}{4}-\frac{1}{p+1}\right)\|\nabla u\|^4_2 + b\left(\frac{1}{2}-\frac{1}{p+1}\right)\|\nabla u\|^2_2
     - \lambda \left(\frac{1}{2}-\frac{1}{p+1}\right)\|u\|^2_2 \nonumber\\
     &\geq&  a\left(\frac{1}{4}-\frac{1}{p+1}\right)\|\nabla u\|^4_2 + \left(\frac{1}{2}-\frac{1}{p+1}\right)b_0 \|\nabla u\|_2^2.
  \end{eqnarray}
 Since $3<p<5$, $\lambda \leq b\lambda_1$ and by Lemma 4.4 (3) along with \eqref{4.5}, we infer that $J(u)$ is bounded from below on $\textit{\textbf{N}}_p$ and
 \begin{eqnarray*}
    J(u)\geq a\left(\frac{1}{4}-\frac{1}{p+1}\right) \rho_p^4,\ \ \ {\rm for}\ u\in \textit{\textbf{N}}_p,
 \end{eqnarray*}
where $\rho_p = \left(S_{p+1}^{(p+1)/2}a\right)^{\frac{1}{p-3}}$ is chosen in Lemma 4.4. Therefore,
 \begin{eqnarray}\label{4.6}
 d_p \geq \left(\frac{1}{4}-\frac{1}{p+1}\right) a \rho_p^4 >0.
 \end{eqnarray}

   Next, we show that $d_p < +\infty$. Put $\phi_\Lambda \in H^1_0(\Omega)$ satisfying \eqref{1.9}, and let $\hat{\tau}=\tau_{\phi_\Lambda}>0$ be the number determined by Lemma 4.1 (i) with respect to $u=\phi_\Lambda$, such that $\hat{\tau} \phi_\Lambda \in \textit{\textbf{N}}_p$  and
  \begin{eqnarray}\label{4.7}
    \hat{\tau}^{p-1} \|\phi_\Lambda\|_{p+1}^{p+1} &=& a\hat{\tau}^2 \|\nabla \phi_\Lambda\|_2^4 + b \|\nabla \phi_\Lambda\|_2^2 - \lambda \|\phi_\Lambda\|_2^2 \nonumber\\
     &=& a \Lambda \hat{\tau}^2 + b\sqrt{\Lambda} - \lambda \|\phi_\Lambda\|_2^2.
   \end{eqnarray}
From $p>3$ and H\"{o}lder inequality, we have $\|\phi_\Lambda\|_4^4 \leq |\Omega|^{(p-3)/(p+1)} \|\phi_\Lambda\|_{p+1}^{4}$, and $\|\phi_\Lambda\|_2^2 \leq |\Omega|^{1/2} \|\phi_\Lambda\|_{4}^{2}$. Noting $\|\phi_\Lambda\|_4^4=1$, then we deduce from \eqref{4.7} that
\begin{eqnarray*}
  |\Omega|^{-\frac{p-3}{4}}  \hat{\tau}^{p-1} \leq  a \Lambda \hat{\tau}^2 + b\sqrt{\Lambda}+c_0 ,
\end{eqnarray*}
where
 \begin{eqnarray}\label{4.8}
 c_0 =\left\{\begin{array}{ll}
      0,\ \ &{\rm for}\ 0\leq \lambda \leq b\lam_1,
    \\[1.2ex]
   -\lam |\Omega|^{1/2} , \ \ &{\rm for}\ \lambda < 0.
      \end{array}\right.
 \end{eqnarray}
Thus, by virtue of $3<p<5$, there exists $\gamma_0>0$ such that $\hat{\tau}\leq \gamma_0$. Now observing that $\hat{\tau} \phi_\Lambda \in \textit{\textbf{N}}_p$  and
\begin{eqnarray*}
  &&J(\hat{\tau} \phi_\Lambda)\\
   &=& a\hat{\tau}^4 \left(\frac{1}{4}-\frac{1}{p+1}\right)\|\nabla  \phi_\Lambda\|^4_2 + b\hat{\tau}^2 \left(\frac{1}{2}-\frac{1}{p+1}\right)\|\nabla \phi_\Lambda\|^2_2
     - \lambda \hat{\tau}^2 \left(\frac{1}{2}-\frac{1}{p+1}\right)\|\phi_\Lambda\|^2_2\\
   &\leq& a \gamma_0^4  \Lambda \left(\frac{1}{4}-\frac{1}{p+1}\right) +  b \gamma_0^2 \sqrt{\Lambda} \left(\frac{1}{2}-\frac{1}{p+1}\right) + c_0 \gamma_0^2 \left(\frac{1}{2}-\frac{1}{p+1}\right),
\end{eqnarray*}
where $c_0$ is defined by \eqref{4.8}. Therefore, we arrive at $d_p = \inf\limits_{u \in \textit{\textbf{N}}_p} J(u) <+\infty$.\par
  By Proposition 4.2 and a standard argument as in Willem \cite{Wi} (Chapter 4), we obtain \eqref{4.4}.
  \end{proof}

 In order to prove the non-emptiness of the potential well $\textit{\textbf{W}}_p^{\,+} = J^{d_p} \cap \textit{\textbf{N}}_p^{\,+}$ and the unstable set $\textit{\textbf{W}}_p^{\,-} = J^{d_p} \cap \textit{\textbf{N}}_p^{\,-}$ for $3<p<5$, we choose
 \begin{eqnarray}\label{4.9}
 R_p =\left[\frac{4(p+1)d_p}{(p-3)a}\right]^{1/4} \ \ {\rm and}\ \hat{r}_p = \min\left\{\rho_p,\ \left(-\frac{c_1}{a} + \frac{1}{a}\sqrt{c_1^2 + 4ad_p}\right)^{1/2}\right\},
\end{eqnarray}
 where $\rho_p>0$ is given in Lemma 4.4, and $c_1$ appears in \eqref{1.7}. By \eqref{4.6} we have $R_p\geq \rho_p$ for $3<p<5$ and $\lam \leq b\lam_1$.\par
  We have:\par
 \begin{thm}
  If  $3<p<5$, $\lambda \leq b\lambda_1$, then $\textit{\textbf{W}}_p^{\,+}\neq \emptyset$, $\textit{\textbf{W}}_p^{\,-}\neq \emptyset$ and
\begin{eqnarray}
 && B_{\hat{r}_p} \subseteq \textit{\textbf{W}}_p^{\,+} \subseteq B_{R_p}\cap J^{d_p},\label{4.10}\\
 && \bar{B}_{R_p}^{\,c}\cap J^{d_p} \subseteq \textit{\textbf{W}}_p^{\,-} \subseteq \bar{B}_{\rho_p}^{\,c}\cap J^{d_p},\label{4.11}
\end{eqnarray}
where $R_p$, $\hat{r}_p$ are defined by $\eqref{4.9}$, and $\rho_p$ is given in Lemma $4.4$.
\end{thm}
\begin{proof}
Applying Lemma 4.4 (1), we know that $B_{\hat{r}_p} \subseteq \textit{\textbf{N}}_p^{\,+}$. Further, by \eqref{1.6}, \eqref{zh809-4} and \eqref{1.13}, we have
 \begin{eqnarray*}
     J(u) \leq \frac{a}{4}\|\nabla u\|_2^4 + \frac{c_1}{2}\|\nabla u\|_2^2 < d_p,\ \ {\rm for}\ u\in B_{\hat{r}_p},
 \end{eqnarray*}
and
 \begin{eqnarray}
    d_p > J(u) &\geq&  a\left(\frac{1}{4}-\frac{1}{p+1}\right)\|\nabla u\|^4_2 + b_0\left(\frac{1}{2}-\frac{1}{p+1}\right) \|\nabla u\|_2^2 + \frac{1}{p+1} I(u)\\
         &\geq& a\left(\frac{1}{4}-\frac{1}{p+1}\right)\|\nabla u\|^4_2,\ \ \ \ \qquad {\rm for}\ u\in \textit{\textbf{W}}_p^{\,+}.  \nonumber
  \end{eqnarray}
Hence, we obtain $\textit{\textbf{W}}_p^{\,+}\neq \emptyset$ and \eqref{4.10}. \par
 To prove \eqref{4.11}, we first check that $\bar{B}_{R_p}^{\,c}\cap J^{d_p} \neq \emptyset$. Choosing $\psi \in H^1_0(\Omega)$ and $\|\nabla \psi\|_2 =1$, it follows from Proposition 4.2 that, there exists a unique number $\tau_\psi >0$, such that $J(\tau \psi)$ is increasing for $\tau \in (0, \tau_\psi)$, and is decreasing for $\tau \in (\tau_\psi, +\infty)$.  Noting
\begin{eqnarray*}
J(\tau \psi) = \frac{a \tau^4}{4} + \frac{b \tau^2}{2} - \frac{\lambda \tau^2}{2}  \|\psi\|_2^2 - \frac{\tau^{p+1}}{p+1} \|\psi\|_{p+1}^{p+1},
\end{eqnarray*}
 then $\lim\limits_{\tau\rightarrow 0} J(\tau \psi) = 0$, and $\lim\limits_{\tau\rightarrow +\infty} J(\tau \psi) = -\infty$ by virtue of $p>3$. Thus, we deduce from the above facts that there exist $\eta_1 >0$, such that $J(\tau \psi)<d_p$ for all $\tau >\eta_1$. Selecting $\tau_1 >\max \left\{R_p, \eta_1\right\}$, then we have $\tau_1 \psi \in J^{d_p}\cap \bar{B}_{R_p}^{\,c}$ and $\bar{B}_{R_p}^{\,c}\cap J^{d_p} \neq \emptyset$.\par
  Let $u\in \bar{B}_{R_p}^{\,c}\cap J^{d_p}$, then we deduce from \eqref{1.6} and (4.12) that
   $$ d_p > J(u) \geq a\left(\frac{1}{4}-\frac{1}{p+1}\right)R_p^4 + \frac{1}{p+1} I(u).
   $$
 Due to \eqref{4.9} and Lemma 4.4 (2), we obtain $\textit{\textbf{W}}_p^{\,-}\neq \emptyset$ and \eqref{4.11}.
\end{proof}

\section{Proofs of main results}
\setcounter{equation}{0}
\subsection{Invariant sets under the flow of \eqref{zh809-1}}
\label{}\ \ \ \
  To begin with, we prepare the following energy identity which will be used frequently in proving our results.
  \begin{prop}
    Let $u\in H(T)$ be a solution for $\eqref{zh809-1}$, where $H(T)$ is defined by \eqref{2.1}, then
    \begin{eqnarray}\label{5.1}
 E(t) :=  \frac{1}{2}\|u_t (t)\|_2^2 + J(u(t))= E(0),\quad for\ any\ \ t\in (0,T).
 \end{eqnarray}
  \end{prop}
\begin{proof}
  Using \eqref{zh809-1}, a direct calculation shows that
  \begin{eqnarray*}
  \frac{\dif}{\dif t}E(t) = \int_{\Omega} \left[ u_{tt}-\left(a \int_\Omega |\nabla u|^2 \dif x +b\right)\Delta u - \lambda u - |u|^{p-1}u \right] u_t \dif x =0.
  \end{eqnarray*}
  Then $E(t) = E(0)$ holds for any $t\in (0,T)$.
\end{proof}

 Now we are in a position to study the invariance of the stable set $\textit{\textbf{W}}_p^{\,+}$ and unstable set $\textit{\textbf{W}}_p^{\,-}$ under the flow of \eqref{zh809-1} for $3\leq p<5$.\\[0.4em]
{\bf{Proof of Theorem 2.1.}} \emph{Case $A$}: $p=3$ and $\lam<b\lam_1$.\par
 $(A1)$ Let $u \in H(T_{{\rm{max}}})$ be the unique solution of $\eqref{zh809-1}$ with $u_0 \in \textit{\textbf{W}}_3^{\,+}$ and $E(0)< d_3$. Suppose by contradiction that there exists $t_0 \in (0, T_{{\rm{max}}})$ such that $u(t) \in \partial\textit{\textbf{W}}_3^{\,+}$, then we have either (i) $J(u(t_0)) = d_3$, or (ii) $I(u(t_0)) = 0$ and $\|\nabla u(t_0)\|_2 \neq 0$.\par
  From the assumption that $E(0)<d_3$ and the energy identity $E(t)=E(0)$ for $t>0$, we have
  \begin{eqnarray}\label{5.2}
    \frac{1}{2}\|u(t_0)\|_2^2 + J(u(t_0)) = E(0) < d_3.
  \end{eqnarray}
Thus (i) is impossible.\par
  Now we assume that (ii) holds. Hence $u(t_0) \in \textit{\textbf{N}}_3$, and by Proposition 3.5 we have $J(u(t_0)) \geq d_3$. This contradicts \eqref{5.2}, and we obtain $(A1)$.

$(A2)$ By contradiction we assume that there is $t_1 \in (0, T_{{\rm{max}}})$ such that $u(t) \in \textit{\textbf{W}}_3^{\,-}$ on $[0, t_1)$, and $u(t_1) \not\in \textit{\textbf{W}}_3^{\,-}$. Then, it follows from the continuity of $J(u(t))$ and $I(u(t))$ in $t$ that either\par
  (i) $J(u(t_1) = d_3$, \quad or \quad (ii) $I(u(t_1)) = 0$. \par
  To obtain $(A2)$, it suffices for us to prove that neither of the cases (i), (ii) are true.  Noting that
   \begin{eqnarray*}
    \frac{1}{2}\|u(t_1)\|_2^2 + J(u(t_1)) = E(0) < d_3,
  \end{eqnarray*}
  we check that (i) is impossible.\par
  Now suppose that (ii) is true, that is $I(u(t_1)) = 0$, and $I(u(t)) < 0$ for all $t \in [0, t_1)$. Hence, we deduce from Lemma 3.4 (2) that
   $$\|\nabla u(t)\|_2 > \rho_3\ {\rm for}\ t \in [0, t_1){\rm ,}\quad {\rm and}\quad \|\nabla u(t_1)\|_2 \geq \rho_3 >0.
   $$
   Then $u(t_1) \in \textit{\textbf{N}}_3$. By Proposition 3.5, we have $J(u(t_1)) \geq d_3$, which is also impossible.\par\medskip

  \emph{Case $B$}: $p=3$ and $b\lam_1 <\lam < b\lam_1 + \delta$.\par
  Suppose by contradiction that there exists $t_1 \in (0, T_{{\rm{max}}})$ such that the solution $u(t) \in \textit{\textbf{N}}_3^{\,-} \cap J^{d_3^{\,-}}$ on $[0, t_1)$, and $u(t_1) \not\in \textit{\textbf{N}}_3^{\,-} \cap J^{d_3^{\,-}}$. Then, by the continuity of $J(u(t))$ and $I(u(t))$ in $t$ we obtain that either
  (i) $J(u(t_1) = d_3^{\,-}$, or (ii) $I(u(t_1)) = 0$. \par
  First, it follows from $E(0) < d_3^{\,-}$ that (i) is impossible.
  Now suppose that (ii) is true, and noting that $\textit{\textbf{N}}_3 \cap L^0 = \emptyset$, then one of the following cases holds:\par
   (a) $\|\nabla u(t_1)\|_2 =0$, (b) $u(t_1) \in \textit{\textbf{N}}_3 \cap L^+$, (c) $u(t_1) \in \textit{\textbf{N}}_3 \cap L^-$. \par
  If (a) holds, then $J(u(t_1))=0> d_3^{\,-}$, which is impossible as $E(0) < d_3^{\,-}$.\par
  If (b) holds, then we deduce from Lemma 3.11 (1) that $J(u(t_1))\geq d_3^{\,+} > d_3^{\,-}$, which contradicts $E(0) < d_3^{\,-}$.\par
  If (c) holds, then by the definition of $d_3^{\,-}$ we get $J(u(t_1)) \geq d_3^{\,-}$, which is also impossible.\par
  Hence (ii) is false, and we finish the proof of Case B.\par
  Finally, Case C can be deduced by Lemma 4.4 (2), Proposition 4.5 and the similar argument as in Case A.\hfill  $\Box$

\subsection{The 4-sublinear case}
\label{}\ \ \ \
  We prove Theorem 2.2 via the energy estimates for $1<p<3$.\\[0.4em]
{\bf{Proof of Theorem 2.2.}} Let $u \in H(T_{{\rm{max}}})$ be a solution for $\eqref{zh809-1}$ with $1<p<3$, and $T_{{\rm{max}}}$ is the maximal existence time of $u$. In view of \eqref{1.8}, we have
  \begin{eqnarray}\label{5.3}
  E(t) \geq \frac{1}{2}\|u_t (t)\|_2^2 + \frac{b}{2}\|\nabla u (t)\|_2^2 + \frac{a}{4}\|\nabla u (t)\|_2^4 - \frac{\hat{\lam}}{2\lam_1}\|\nabla u (t)\|_2^2 - \frac{S_{p+1}^{-\,(p+1)/2}}{p+1}\|\nabla u (t)\|_2^{p+1},
  \end{eqnarray}
for $0<t<T_{{\rm{max}}}$, where $\hat{\lam} = \max\{\lam, 0\}$. Let
\begin{eqnarray*}
 \varphi(s) = \frac{a}{4}s^4 - \frac{\hat{\lam}}{2\lam_1}s^2 - \frac{1}{(p+1)S_{p+1}^{(p+1)/2}}s^{p+1}, \ \ {\rm for}\ s\geq 0.
\end{eqnarray*}
By virtue of $a>0$ and $1<p<3$, we know that the function $\varphi$ is bounded from below. \par

Taking $\varphi_0 = \inf\limits_{s\geq 0} \varphi(s)$, then it follows from \eqref{5.3} and the fact of $E(t) = E(0)$ that
\begin{eqnarray*}
  \frac{1}{2}\|u_t (t)\|_2^2 + \frac{b}{2}\|\nabla u(t)\|_2^2 \leq E(0) - \varphi_0,\ \ {\rm for}\ 0<t<T_{{\rm{max}}},
\end{eqnarray*}
which proves \eqref{2.2}.\par
  Now, we turn to prove \eqref{2.3}. Let $(u_0, u_1)$ satisfy $\|\nabla u_0\|_2 >0$ and $E(0) \leq 0$, then we deduce from \eqref{1.6}, \eqref{1.8} and Proposition 5.1 that
  \begin{eqnarray}\label{5.4}
      \left[\frac{b_0}{2} - \frac{S_{p+1}^{-(p+1)/2}}{p+1}\|\nabla u(t)\|_2^{p-1}\right] \|\nabla u(t)\|_2^2 \leq E(0) \leq 0, \qquad {\rm for\ all}\ \ t\in [0, T_{{\rm{max}}}).
  \end{eqnarray}

 To obtain \eqref{2.3}, we suppose by contradiction that there exists some $\bar{t}\in (0, T_{{\rm{max}}})$, such that $\|\nabla u(\bar{t})\|_2 =0$. Then from $\|\nabla u_0\|_2 >0$ and the continuity of $u(t)$ at $\bar{t}$, we deduce that there exists $t_1 \in (0, \bar{t})$ such that
 $$0< \|\nabla u(t_1)\|_2 < \left[\frac{1}{2}b_0 (p+1) S_{p+1}^{(p+1)/2}\right]^{1/(p-1)},
 $$
which contradicts with \eqref{5.4}.\par
 Thus, $\|\nabla u(t)\|_2 >0$ for all $t\in [0, T_{{\rm{max}}})$, and by virtue of \eqref{5.4}, we obtain \eqref{2.3}.
This completes the proof of Theorem 2.2.   \hfill  $\Box$

\subsection{The asymptotically 4-linear case}
 \label{}\ \ \ \   For $p=3$, we start by proving the boundedness and asymptotically behavior of the solution for \eqref{zh809-1}.\\[0.4em]
{\bf{Proof of Theorem 2.4.}} Let $u \in H(T_{{\rm{max}}})$ be a solution for $\eqref{zh809-1}$ with the maximal existence time $T_{{\rm{max}}}>0$.\par
\emph{Case $(H1)$}: $a> 1/\Lambda$\ \  and\ \  $\lambda \in \mathbb{R}$.\par
    For $p=3$, we know by \eqref{zh809-4} and \eqref{5.1} that
    \begin{eqnarray}\label{5.5}
      E(0) &=& \frac{1}{2}\|u_t (t)\|_2^2 + \frac{a}{4}\|\nabla u (t)\|_2^4 + \frac{b}{2}\|\nabla u (t)\|_2^2 - \frac{\lam}{2}\|u (t)\|_2^2 - \frac{1}{4}\|u (t)\|_4^{4} \nonumber \\
       &\geq& \frac{1}{2}\|u_t (t)\|_2^2 + \frac{b}{2}\|\nabla u (t)\|_2^2 + \frac{1}{4}\left(a-\frac{1}{\Lambda}\right) \|\nabla u (t)\|_2^4 -  \frac{\hat{\lam}}{2\lam_1}\|\nabla u (t)\|_2^2,
    \end{eqnarray}
    where $\hat{\lam} = \max\{\lam, 0\}$, $0<t<T_{{\rm{max}}}$.
    Noting that $a> 1/\Lambda$, the function
    $$\hat{h}(s) = \frac{1}{4}\left(a-\frac{1}{\Lambda}\right) s^2 -  \frac{\hat{\lam}}{2\lam_1}s,\ \ s \geq 0
    $$
     attains its minimum
      \begin{eqnarray*}
      h_1 = \min\limits_{s\geq 0} \hat{h}(s) =\left\{\begin{array}{ll}
      - \frac{\lam^2 \Lambda}{4\lam_1^2 \left(a\Lambda-1\right)},\ \ &{\rm for}\ \lambda >0,
    \\[1.2ex]
   0 , \ \ &{\rm for}\ \lambda \leq 0.
   \end{array}\right.
    \end{eqnarray*}
It follows from \eqref{5.5} that
$$\|u_t (t)\|_2^2 + b\|\nabla u (t)\|_2^2 \leq 2E(0)-2h_1,\quad {\rm for}\quad 0<t<T_{{\rm{max}}}.
$$
Then we obtain (2.4).\par

  \emph{Case $(H2)$}: $a= 1/\Lambda$\ \  and\ \  $\lambda < b\lambda_1$.\par
  Combining with \eqref{1.6} and \eqref{5.5}, we have
  $$E(0) \geq \frac{1}{2}\|u_t (t)\|_2^2 + \frac{b_0}{2} \|\nabla u (t)\|_2^2,\quad  {\rm for}\  t\in [0, T_{{\rm{max}}}).
  $$
Recalling that $b_0 >0$ for $\lambda < b\lambda_1$, we get (2.4).\par

   \emph{Case $(H3)$}: $0<a<1/\Lambda$,\ \ $\lambda < b\lambda_1$,\ \ $u_0 \in \textit{\textbf{W}}_{3} ^{\, +}$ and $E(0)<d_3$.\par
   From Theorem 2.1, we deduce that $u(t)\in \textit{\textbf{W}}_{3} ^{\, +}$ for $t\in [0, T_{{\rm{max}}})$, where $\textit{\textbf{W}}_{3} ^{\, +} = J^{d_3} \cap \textit{\textbf{N}}_{3} ^{\, +}$. Then by $I(u(t))\geq 0$, and combining with \eqref{1.6} and \eqref{1.13} for $p=3$, we have
   \begin{eqnarray}\label{5.6}
     J(u(t)) \geq   \frac{b_0}{4} \|\nabla u(t)\|_2^2, \quad {\rm for}\quad 0<t<T_{{\rm{max}}},
  \end{eqnarray}
 where $b_0$ appears in \eqref{1.7}. In view of $E(0)<d_3$, \eqref{5.1} and \eqref{5.6},
 we obtain (2.4).

   Finally, we prove \eqref{2.6}.
It follows from \eqref{2.5} and \eqref{3.11} that $\psi_1 \in S$, and
 \begin{eqnarray*}
   d_3  &=& \inf\limits_{u \in S} \frac{\left(b\|\nabla u\|_2^2 - \lambda \|u\|_2^2\right)^2}{4 \left(\|u\|_{4}^{4} - a\|\nabla u\|_2^4\right)}\\
   &\leq&  \frac{\left(b\|\nabla \psi_1\|_2^2 - \lambda \|\psi_1\|_2^2\right)^2}{4 \left(\|\psi_1\|_{4}^{4} - a\|\nabla \psi_1\|_2^4\right)}\\
   &=&  \frac{(b\lam_1 - \lam)^2 \|\psi_1\|_2^{4}}{4\left(\|\psi_1\|_{4}^{4} - a\|\nabla \psi_1\|_2^4\right)}.
   \end{eqnarray*}
We find $d_3 \rightarrow 0$ as $\lam \rightarrow b\lam_1$. Hence, by virtue of \eqref{5.1} and \eqref{5.6}, we can see that
 $$\frac{1}{2}\|u_t (t)\|_2^2 + \frac{b_0}{4} \|\nabla u (t)\|_2^2 \leq E(0) <d_3 \rightarrow 0,
 $$
as $\lam \rightarrow b\lam_1$, for any $t\in [0, T_{{\rm{max}}})$. Then we arrive at \eqref{2.6} and finish the proof of Theorem 2.4. \hfill $\Box$\par\smallskip

    Next, we turn to show the blow-up of the solution for \eqref{zh809-1} with $p=3$ and $\lam <b\lam_1$.\\[0.4em]
{\bf{Proof of Theorem 2.6.}}
  Assume for contradiction that $u \in H(T_{{\rm{max}}})$ is a global solution of \eqref{zh809-1} with $p=3$, $0<a<1/\Lambda$ and $\lam < b\lam_1$, that is, the maximal existence time of $u$ is $T_{{\rm{max}}} = \infty$. Define
 $$M(t) = \|u(t)\|_2^2,
 $$
 then $M \in C^2 (0, +\infty)$, and
 \begin{eqnarray}\label{5.7}
   M'(t) &=& 2 \int_{\Omega} uu_t \dif x, \nonumber\\
    M''(t) &=& 2 \|u_t (t)\|_2^2 - 2a\|\nabla u (t)\|_2^4 - 2b \|\nabla u (t)\|_2^2 + 2\lam \|u (t)\|_2^2 + 2\|u (t)\|_4^4.
 \end{eqnarray}

  \emph{Case $(h1)$}: $E(0)<0$.\par
 We infer from \eqref{5.1} and \eqref{5.7} that
 \begin{eqnarray}\label{5.8}
   M''(t) = 6 \|u_t (t)\|_2^2 + 2b \|\nabla u (t)\|_2^2 - 2\lam \|u (t)\|_2^2 -8E(0).
 \end{eqnarray}
 Then, by virtue of \eqref{1.6}, we have
 \begin{eqnarray}\label{5.9}
   M''(t) \geq 6 \|u_t (t)\|_2^2 + 2b_0 \lam_1 M(t) -8E(0).
 \end{eqnarray}

Noting that $b_0 \geq 0$ for $\lam < b\lam_1$, we get $M''(t) \geq -8E(0)$ and $M'(t) \geq -8E(0) t + M'(0)$. Since $E(0)<0$, we infer that there exists $t_0 >0$, such that $M'(t) >0$ and $M(t) >0$, for $t \geq t_0$. Then,
 \begin{eqnarray}\label{5.10}
   \left(M^{-1/2}\right)'(t_0) = - \frac{1}{2M^{3/2}(t_0)} M'(t_0) <0.
 \end{eqnarray}

   Furthermore, we deduce from \eqref{5.9} that
   \begin{eqnarray*}
     MM''-\frac{3}{2}\left(M'\right)^2 \geq 6\|u\|_2^2 \|u_t\|_2^2 - 8 E(0) \|u\|_2^2 -6 \left(\int_\Omega uu_t \dif x\right)^2 >0,
   \end{eqnarray*}
 and
  \begin{eqnarray}\label{5.11}
    \left(M^{-1/2}\right)''(t) = -\frac{1}{2M^{5/2}} \left[MM'' - \frac{3}{2}\left(M'\right)^2\right] <0,\qquad {\rm for}\ t\geq t_0.
  \end{eqnarray}
Combining with \eqref{5.10} and \eqref{5.11}, we have
  \begin{eqnarray}\label{5.12}
    \left(M^{-1/2}\right)'(t) \leq \left(M^{-1/2}\right)'(t_0) <0,\ \ {\rm for}\ t>t_0.
  \end{eqnarray}
 It follows from \eqref{5.11} and \eqref{5.12} that there exists $T_1 >0$, such that $M^{-1/2}(t) \rightarrow 0$ and $M(t) \rightarrow \infty$, as $t\rightarrow T_1$. Thus, we reach a contradiction, which implies that $T_{{\rm{max}}} < \infty$.\par\smallskip
 \emph{Case $(h2)${\rm :} $E(0)=0$ and} $\int_\Omega u_0 u_1 \dif x > 0$.\par\smallskip
 By using \eqref{5.9} and $E(0)=0$, we obtain $(M^{-1/2})''(t) \leq 0$ for $t\geq 0$, and
 \begin{eqnarray*}
    \left(M^{-1/2}\right)'(t) \leq \left(M^{-1/2}\right)'(0) = - \frac{1}{\|u_0\|_2^3} \int_\Omega u_0 u_1 \dif x <0.
  \end{eqnarray*}
  Hence, there exists $T_1 >0$, such that $M^{-1/2}(t) \rightarrow 0$ and $M(t) \rightarrow \infty$, as $t\rightarrow T_1$, which contradicts $T_{{\rm{max}}} = \infty$.\par\smallskip
   \emph{Case $(h3)${\rm :} $0< E(0) <d_3$ and} $u_0 \in \textit{\textbf{W}}_3^{\,-}$.\par\smallskip
  By applying Theorem 2.1, we obtain $u(t) \in \textit{\textbf{W}}_3^{\,-}$ for all $t>0$. Combining \eqref{5.8} with \eqref{3.12}, we infer that
  \begin{eqnarray*}
    M''(t) \geq 6\|u_t (t)\|_2^2 +8\eta,
  \end{eqnarray*}
 where $\eta = d_3 - E(0) >0$. Hence,
 \begin{eqnarray*}
   M'(t) \geq 8\eta t + M'(0) = 8\eta t + \int_\Omega u_0 u_1 \dif x,
 \end{eqnarray*}
 and there exists $t_0 >0$, such that $M'(t)>0$ for $t\geq t_0$.
  Then using the same way as in the estimates of \eqref{5.10}-\eqref{5.12}, we reach a contradiction and deduce that $T_{{\rm{max}}} < \infty$.\par\smallskip
   \emph{Case $(h4)${\rm :} $E(0)\geq d_3$, \eqref{2.8} and \eqref{2.9} hold} .\par\smallskip
   We split the proof of case $(h4)$ into two steps. \par
   \emph{Step} 1: We claim that
   \begin{eqnarray}\label{5.13}
     u(t) \in \textit{\textbf{N}}_3^{\,-},\ \ \ {\rm for\ all}\ t\geq 0.
   \end{eqnarray}

   If \eqref{5.13} is false, and noting that $u_0 \in \textit{\textbf{N}}_3^{\,-}$, then there would exist a time $t_1 >0$, such that $u(t_1) \in \textit{\textbf{N}}_3$ and
  \begin{eqnarray}\label{5.14}
    I(u(t)) < 0, \ \ \ {\rm for}\ 0\leq t < t_1.
  \end{eqnarray}
  It follows from \eqref{5.7} and \eqref{5.14} that
  \begin{eqnarray*}
    M''(t) = 2\|u_t (t)\|_2^2 - 2I(u(t)) >0,  \ \ {\rm for}\ 0\leq t < t_1,
  \end{eqnarray*}
  which ensures that $M'(t)$ is strictly increasing on $[0, t_1)$. Then, by $M'(0) = 2\int_\Omega u_0 u_1 \dif x > 0$, we have $M'(t) > M'(0) > 0$ for every $t\in (0, t_1)$. This implies that $M(t)$ is strictly increasing on $[0, t_1)$, and
  \begin{eqnarray*}
   M(t) > M(0) = \|u_0\|_2^2 > \frac{4}{b_0\lambda_1}E(0),\ \ {\rm for}\ 0< t < t_1.
  \end{eqnarray*}
  From the monotonicity of $M(t)$ and the continuity of $u(t)$ at $t_1$, we deduce that
  \begin{eqnarray}\label{5.15}
   \|u(t_1)\|_2^2 = M(t_1) > \frac{4}{b_0\lambda_1}E(0).
  \end{eqnarray}

   On the other hand, by virtue of \eqref{1.6}, \eqref{1.13} and $p=3$, we get
   \begin{eqnarray}\label{5.16}
     E(t) &=& \frac{1}{2}\|u_t (t)\|_2^2 + \frac{b}{4}\|\nabla u (t)\|_2^2 - \frac{\lam}{4}\|u(t)\|_2^2 + \frac{1}{4}I(u(t)) \nonumber\\
     &\geq& \frac{1}{2}\|u_t (t)\|_2^2 + \frac{b_0 \lam_1}{4}\|u (t)\|_2^2 + \frac{1}{4}I(u(t)).
   \end{eqnarray}
Taking $t=t_1$ in \eqref{5.16}, then it follows from  $I(u(t_1))=0$ and Proposition 5.1 that
 $$\|u(t_1)\|_2^2 \leq \frac{4}{b_0 \lam_1} E(0),
 $$
which contradicts with \eqref{5.15}. Hence, we prove \eqref{5.13}.\par
\emph{Step} 2: Next, we prove the blow-up result under the hypothesis of $(h4)$.\par
   By \eqref{5.13} and the proof above, we have $M''(t)>0$ and $M'(t)>0$, for all $t>0$. Thus, we deduce from (h4) that
 \begin{eqnarray}\label{5.17}
    M(t) > M(0) > \frac{4}{b_0\lambda_1}E(0),\ \ {\rm for\ all}\ t>0.
\end{eqnarray}
By virtue of \eqref{5.9}, \eqref{5.17} and \eqref{2.8}, we have $M''(t) \geq 6 \|u_t(t)\|_2^2$. Repeating the same argument as in proving \eqref{5.11} and \eqref{5.12}, we obtain $(M^{-1/2})''(t) \leq 0$ and $(M^{-1/2})'(t) < 0$ for all $t> 0$. Then we reach a contradiction and prove the case $(h4)$ of Theorem 2.6.\par\smallskip
   {\emph {Proof of $\eqref{2.10}$}:\par\smallskip
  Noting that $E(0)\leq 0$, then we apply \eqref{1.6}, \eqref{1.19} and Proposition 5.1 to obtain that
  \begin{eqnarray}\label{5.18}
     \frac{b_0}{2}  \|\nabla u(t)\|_2^2- \frac{1}{4}\left(\frac{1}{\Lambda}-a\right)\|\nabla u(t)\|_2^{4} \leq E(0) \leq 0, \qquad {\rm for\ all}\ \ t\in [0, T_{{\rm{max}}}).
  \end{eqnarray}

 To reach \eqref{2.10}, we suppose by contradiction that there exists some $\bar{t}\in (0, T_{{\rm{max}}})$, such that $\|\nabla u(\bar{t})\|_2 =0$. Then from $\|\nabla u_0\|_2 >0$ and the continuity of $u(t)$ at $\bar{t}$, we deduce that there exists $t_1 \in (0, \bar{t})$ such that
 $$0< \|\nabla u(t_1)\|_2^2 < \frac{2b_0 \Lambda}{1-a\Lambda},
 $$
which contradicts with \eqref{5.18}. Thus, $\|\nabla u(t)\|_2 >0$ for all $t\in [0, T_{{\rm{max}}})$. Then by \eqref{5.18}, we obtain \eqref{2.10} and finish the proof of Theorem 2.6.
\hfill$\Box$\par\smallskip

   We finish this subsection by proving the blow-up of solution for \eqref{zh809-1} when $p=3$ and $b\lam_1 <\lam <b\lam_1 +\delta$.\\[0.3em]
{\bf{Proof of Theorem 2.8.}}  Suppose by contradiction that the maximal existence time of the solution $u\in H(T_{{\rm{max}}})$ is $T_{{\rm{max}}} = \infty$. Then it follows from Theorem 2.1 that $I(u(t)) <0$ for all $t>0$. Furthermore, by Lemma 3.11 (3) we obtain that
\begin{eqnarray}\label{5.19}
  b \|\nabla u(t)\|_2^2 - \lam \|u(t)\|_2^2 > 4d_3^{\,-},\qquad {\rm for\ all}\quad t>0.
\end{eqnarray}

Define $M(t) = \|u(t)\|_2^2$ and $\eta_3 = d_3^{\,-} - E(0)$, then $M \in C^2 (0, +\infty)$ and $\eta_3 >0$. Combining with \eqref{5.8} and \eqref{5.19}, we get
\begin{eqnarray}\label{5.20}
  M''(t) > 6\|u_t(t)\|_2^2 + 8\eta_{3} >0,
\end{eqnarray}
and so $M'(t) \geq 8\eta_3 t + \int_\Omega u_0 u_1 \dif x$. Hence, there exists $t_0 >0$ such that $M'(t) >0$ for all $t>t_0$. In view of \eqref{5.20}, we infer that there exists $t_1 > t_0$ such that $M'(t_1)>0$, $M(t_1)>0$ and $M(t)>0$ for any $t>t_1$. It then follows exactly as in the proof of \eqref{5.10}-\eqref{5.12} that there exists $T_1 >0$, such that $M^{-1/2}(t) \rightarrow 0$ and $M(t) \rightarrow \infty$ as $t\rightarrow T_1$, which is impossible.\par
   Thus, we deduce that $T_{{\rm{max}}} < \infty$.\hfill $\Box$

\subsection{The 4-superlinear case}
\label{}\ \ \ \ First, we show that the solution of \eqref{zh809-1} is bounded uniformly in time for $3<p<5$ and $\lam \leq b\lam_1$, provided that the initial data starting in the stable set and having low energy $E(0)< d_p$.\\[0.3em]
{\bf{Proof of Theorem 2.9.}} Let $u\in H(T_{{\rm{max}}})$ be a solution for $\eqref{zh809-1}$ with the maximal existence time $T_{{\rm{max}}}>0$. Noting that $\lam \leq b\lam_1$, $u_0 \in \textit{\textbf{W}}_{p} ^{\, +}$ and $E(0)<d_p$,
and by virtue of Theorem 2.1, we have $u(t)\in \textit{\textbf{W}}_{p} ^{\, +}$ for $t\in [0, T_{{\rm{max}}})$. Then we deduce from (4.12) and $I(u(t))>0$ that
  \begin{eqnarray}
    a\left(\frac{1}{4}-\frac{1}{p+1}\right)\|\nabla u(t)\|^4_2 \leq J(u(t)) <d_p, \ \ \ {\rm for}\ t\in [0, T_{{\rm{max}}}).
  \end{eqnarray}
Moreover, noting that $E(0)<d_p$ and using energy identity, we have
 \begin{eqnarray}
 E(t) =  \frac{1}{2}\|u_t (t)\|_2^2 + J(u(t))= E(0)<d_p,
 \end{eqnarray}
 which leads to \eqref{2.14}. \hfill $\Box$\par\smallskip

   Now, we study the blow-up behavior and vacuum region of the solution for \eqref{zh809-1} with $p>3$.\\[0.3em]
{\bf{Proof of Theorem 2.10.}}
 Let $u \in H(T_{{\rm{max}}})$ be a solution of \eqref{zh809-1} with $p>3$, and suppose to the contrary that the maximal existence time of $u$ is $T_{{\rm{max}}} = \infty$. Define $M(t) = \|u(t)\|_2^2$, then $M \in C^2 (0, +\infty)$, $M'(t) = 2 \int_{\Omega} uu_t \dif x$, and
 \begin{eqnarray}
 M''(t)  &= & 2 \|u_t (t)\|_2^2 - 2a\|\nabla u (t)\|_2^4 - 2b \|\nabla u (t)\|_2^2 + 2\lam \|u (t)\|_2^2 + 2\|u (t)\|_{p+1}^{p+1}\label{5.23}\\
   &=& (p+3)\|u_t (t)\|_2^2 + \frac{(p-3)a}{2}\|\nabla u (t)\|_2^4 + (p-1)b \|\nabla u (t)\|_2^2\nonumber\\
   && - (p-1)\lam \|u (t)\|_2^2 - 2(p+1)E(0).\nonumber
 \end{eqnarray}

   In order to reach a contradiction, we need to show that there exists $T_1 >0$ such that
   \begin{eqnarray}\label{zh5.25}
      M(t) \rightarrow \infty,\qquad {\rm as}\ t\rightarrow T_1.
   \end{eqnarray}

 \emph{Cases $(a1)$-$(a2)${\rm :}} Similar to the proof of Theorem 2.6, we deduce that there exists $t_0 >0$, such that
 $M''(t) \geq (p+3) \|u_t(t)\|_2^2$, $(M^{-\alpha})''(t) \leq 0$ and $(M^{-\alpha})'(t) < 0$ for all $t> t_0$, where $\alpha = (p-1)/4$. Hence, we get \eqref{zh5.25} which contradicts with $T_{{\rm{max}}} = \infty$. Then we prove the cases $(a1)$-$(a2)$ of Theorem 2.10.\par\smallskip
 \emph{Case $(a3)${\rm :}} We infer from Theorem 2.1 that $I(u(t))<0$ for all $t>0$. Hence, by \eqref{5.23} and repeat the argument as in \eqref{5.10}-\eqref{5.12}, we obtain \eqref{zh5.25}. We then reach a contradiction and derive that $T_{{\rm{max}}} < \infty$.\par\smallskip
 \emph{Case $(a4)${\rm :}} In view of \eqref{5.23}, and using the same way as in proving the case (h4) of Theorem 2.6, we have $u(t) \in \textit{\textbf{N}}_p^{\,-}$ for all $t>0$, which enables us to get \eqref{zh5.25} and $T_{{\rm{max}}} < \infty$.\par\smallskip
 \emph{Case $(b1)${\rm :}} It follows from \eqref{5.23} that
 \begin{eqnarray}\label{zh5.26}
     M''(t) \geq (p+3)\|u_t (t)\|_2^2 + \frac{p-3}{2}a\lam_1^2 M(t)^2 + (p-1)(b\lam_1 -\lam) M(t) - 2(p+1)E(0).
 \end{eqnarray}

 To proceed, we define
 \begin{eqnarray*}
   h(s) = \frac{p-3}{2}a\lam_1^2 s^2 - (p-1)(\lam - b\lam_1) s.
 \end{eqnarray*}
By $\lam >b\lam_1$, then a direct calculation shows that $h(s)$ attains its minimum at the point
 $$s_0 = \frac{(p-1)(\lam - b\lam_1)}{(p-3)a\lam_1^2},
 $$
 and $h(s_0) =  -h_0$, where $h_0$ is appearing in \eqref{2.16}.

   Noting that $-2(p+1)E(0) > h_0$, then we deduce from \eqref{zh5.26} that $M''(t) > (p+3)\|u_t\|_2^2$ and $M''(t) \geq -2(p+1)E(0) - h_0 >0$ for $t>0$. By the same arguments that
 we used in proving the Case $(a1)$, we get \eqref{zh5.25} and show that $T_{{\rm{max}}} < \infty$.\par\smallskip
 \emph{Case $(b2)${\rm :}} We infer from \eqref{2.8} that
 \begin{eqnarray}\label{zh5.27}
   \left(M^{-\alpha}\right)'(0)= - \frac{\alpha}{M^{\alpha +1}(0)}M'(0) <0,\qquad {\rm where }\ \alpha = \frac{p-1}{4}.
 \end{eqnarray}

On the other hand, it follows from  \eqref{zh5.26} and $-2(p+1)E(0) = h_0$ that $M''(t) \geq (p+3)\|u_t\|_2^2$. Combining this with \eqref{zh5.27}, we obtain $(M^{-\alpha})''(t) \leq 0$ and $(M^{-\alpha})'(t) < 0$ for all $t> 0$. Thus, we have \eqref{zh5.25}, which proves Case $(b2)$.\par\smallskip
{\emph {Proof of \eqref{2.17}}: By virtue of \eqref{1.6} and \eqref{1.8}, we deduce that
 \begin{eqnarray*}
   E(t) \geq \frac{a}{4}\|\nabla u\|_2^4 - \frac{1}{(p+1)S_{p+1}^{(p+1)/2}}\|\nabla u\|_2^{p+1}.
 \end{eqnarray*}
 Moreover, applying Proposition 5.1 and observing that $E(0)\leq 0$, we have
 \begin{eqnarray*}
   \frac{a}{4}\|\nabla u\|_2^4 \leq \frac{1}{(p+1)S_{p+1}^{(p+1)/2}}\|\nabla u\|_2^{p+1}.
 \end{eqnarray*}
 The rest of the proof is similar to the derivation of \eqref{2.3} in Theorem 2.2.  \hfill $\Box$

%

\section{Compatibility of the assumptions for main theorems}
 \setcounter{equation}{0}
\label{}\ \ \ \     We devote this section to prove  the assumptions of main theorems are  compatible respectively.\par\smallskip

 \emph{Existence of initial data leading to Theorem }2.2:

 \begin{prop}
     Let $1<p<3$, $\lam <b\lam_1$, and suppose that $a>0$ is sufficiently small, then there exists such initial data $(u_0, u_1)$ which satisfy $\eqref{1.2}$, $E(0)\leq 0$ and $\|\nabla u_0\|_2 >0$.
   \end{prop}
\begin{proof}
  Let
   $$\sigma= \inf_{\mbox{\tiny$\begin{array}{c}
   u\in H^1_0(\Omega)\cap H^2 (\Omega)\\
       \|\nabla u\|_2 \neq 0
       \end{array}$}} \frac{~\|\nabla u\|_2^2~}{\|u\|^2_{p+1}},
   $$
   and suppose that $a>0$ is sufficiently small such that
   $$\left[\frac{3-p}{c_1(p+1)}\right]^{\frac{3-p}{p+1}} \left[\frac{2p-2}{a(p+1)}\right]^{\frac{p-1}{p+1}} >\sigma,
   $$
   where $c_1$ is appearing in \eqref{1.7}, then we may choose $v_0 \in H_0^1(\Omega) \cap H^2(\Omega)$ satisfying $\|\nabla v_0\|_2 >0$, and
   \begin{eqnarray}\label{6.1}
  \frac{~\|\nabla v_0\|_2^2~}{\|v_0\|^2_{p+1}} < \left[\frac{3-p}{c_1(p+1)}\right]^{\frac{3-p}{p+1}} \left[\frac{2p-2}{a(p+1)}\right]^{\frac{p-1}{p+1}}.
   \end{eqnarray}

  Let $v_1 \in H_0^1(\Omega)$ with $\|v_1\|_2 =1$, and we take $(u_0, u_1) = (kv_0, m v_1)$, where $k$, $m>0$ will be chosen appropriately later such that $(u_0, u_1)$ satisfy $E(0)\leq 0$. In fact, let
   $$\phi(k) = \frac{a k^2}{4} \|\nabla v_0\|_2^4 - \frac{k^{p-1}}{p+1}  \|v_0\|_{p+1}^{p+1} + \frac{1}{2} \left(b\|\nabla v_0\|_2^2 - \lam\|v_0\|_2^2\right),
   $$
then by virtue of $1<p<3$, \eqref{6.1} and
$$\phi'(k) = \frac{a k}{2} \|\nabla v_0\|_2^4 - \frac{(p-1)k^{p-2}}{p+1}  \|v_0\|_{p+1}^{p+1}.
$$
we can see that $\phi(k)$ attains its infimum at the point
$$k_0 = \left[\frac{2(p-1)\|v_0\|_{p+1}^{p+1}}{a(p+1)\|\nabla v_0\|_2^4}\right]^{\frac{1}{3-p}},
$$
and
\begin{eqnarray*}
  \phi(k_0) &=& -\frac{3-p}{2p+2} \left[\frac{2p-2}{a(p+1)}\right]^{\frac{p-1}{3-p}} \|v_0\|_{p+1}^{\frac{2p+2}{3-p}} \|\nabla v_0\|_{2}^{-\frac{4(p-1)}{3-p}} + \frac{1}{2} \left(b\|\nabla v_0\|_2^2 - \lam\|v_0\|_2^2\right)\\[0.2em]
  &\leq& \frac{1}{2} \|\nabla v_0\|_{2}^{-\frac{4(p-1)}{3-p}} \left(-\frac{3-p}{p+1} \left[\frac{2p-2}{a(p+1)}\right]^{\frac{p-1}{3-p}} \|v_0\|_{p+1}^{\frac{2p+2}{3-p}} + c_1 \|\nabla v_0\|_{2}^{\frac{2p+2}{3-p}} \right)\\[0.2em]
  &<& 0.
\end{eqnarray*}
Hence, $J(k_0 v_0) = k_0^2 \phi(k_0) <0$. Selecting $m=m_1 < \sqrt{-J(k_0 v_0)}$, then for $(u_0, u_1) = (k_0 v_0, m_1 v_1)$,
 $$E(0) = \frac{m_1 ^2}{2} + J(k_0 v_0) <0,
 $$
 and we finish the proof of Proposition 6.1.
 \end{proof}

\emph{Existence of initial data leading to Theorem }2.4:
 \begin{prop}
     Let $p=3$, and assume that $0<a<1/\Lambda$, $\lam <b\lam_1$, then there exists such initial data $(u_0, u_1)$ which satisfy $\eqref{1.2}$, $E(0)< d_3$, $u_0 \in \textit{\textbf{W}}_{3} ^{\, +}$.\par
     Moreover, suppose that
 \begin{eqnarray}\label{6.2}
   a< \frac{1}{\lam_1^2 |\Omega|},
 \end{eqnarray}
 then $\eqref{2.5}$ holds.
   \end{prop}
 \begin{proof}
 Recalling Theorem 3.7 and choosing $u_0 \in H_0^1(\Omega) \cap H^2(\Omega)$ which satisfies $0< \|\nabla u_0\| \leq \hat{r}_3$ with $\hat{r}_3$ being defined by \eqref{3.13}, then we have $u_0 \in \textit{\textbf{W}}_{3} ^{\, +} = J^{d_3} \cap \textit{\textbf{N}}_3^{\,+}$. Let $u_1 \in H_0^1(\Omega)$ with $\|u_1\|_2^2 < d_3 -J(u_0)$, then we obtain that $(u_0, u_1)$ satisfying \eqref{1.2}, $u_0 \in \textit{\textbf{W}}_{3} ^{\, +}$ and $E(0)< d_3$.\par
 Next, we show that \eqref{6.2} leads to \eqref{2.5}. In fact, noting $\|\nabla \psi_1\|_2^2 = \lam_1 \|\psi_1\|_2^2$ and
   $$\|\psi_1\|_2^2 \leq |\Omega|^{1/2} \|\psi_1\|_4^2,
   $$
   then we infer that
   \begin{eqnarray*}
     \|\psi_1\|_4^4 - a \|\nabla \psi_1\|_2^4 \geq \frac{1}{|\Omega|} \|\psi_1\|_2^4 -a \lam_1^2 \|\psi_1\|_2^4 >0,
   \end{eqnarray*}
which implies \eqref{2.5}.
 \end{proof}

 \emph{Existence of initial data leading to Theorem }2.6:
  \begin{prop}
     Suppose that $p=3$, $\lam <b\lam_1$ and $\eqref{2.7}$ hold, then there exists such initial data $(u_0, u_1)$ which satisfy $\eqref{1.2}$ and each of the conditions of $(h1)$-$(h4)$ in Theorem $2.6$.
   \end{prop}
\begin{proof}
For $v_0 \in  \left(\textit{\textbf{N}}_3 \cup \textit{\textbf{N}}_3^{\,-}\right) \cap H^2(\Omega)$ and $v_1 \in H_0^1(\Omega)$ satisfying $\|v_1\|_2 =1$, we take $(u_0, u_1) = (kv_0, m v_1)$ with $k$, $m \in \mathbb{R}_+$. Then $(u_0, u_1)$ satisfy \eqref{1.2}. In what follows, we choose $k$, $m>0$ such that $(u_0, u_1)$ satisfy one of the conditions $(h1)$-$(h4)$ in Theorem 2.6.\par
  (i) \emph {Construct initial data satisfying $(h1)$}. \par
  Let $\mu_0 = \|v_0\|_4^4 - a\|\nabla v_0\|_2^4$, $\mu_1 = b\|\nabla v_0\|_2^2 - \lam \|v_0\|_2^2$. Due to \eqref{1.18}, \eqref{1.6} and $\lam \leq b\lam_1$, we infer that $\mu_0 >0$ and $\mu_1 \geq 0$. By a direct computation, we get
   \begin{eqnarray}\label{6.3}
     E(0) = \frac{m^2}{2} - \frac{\mu_0}{4}k^4 + \frac{\mu_1}{2}k^2.
   \end{eqnarray}
Hence, fixing $m=1$, we can choose sufficiently large $k$ such that $E(0)<0$ and $(u_0, u_1)$ satisfy $(h1)$.\par
 (ii) \emph {Construct initial data satisfying $(h2)$}. \par
  We may suppose that $\int_{\Omega} v_0 v_1 \dif x >0$, and by virtue of \eqref{6.3}, we check that $(u_0, u_1)$ satisfy $(h2)$ for $m=1$ and $k=\bar{k}$, where $\bar{k}>0$ is a root of $\mu_0 k^4 -2\mu_1 k^2 = 2$.\par
  (iii) \emph {Construct initial data satisfying $(h3)$}. \par
   Due to \eqref{1.12} and $p=3$, we obtain that $I(u_0) = I(kv_0) = - k^4 \mu_0 + k^2 \mu_1 <0$ by choosing $k^2 >\mu_1/\mu_0$.
  Since $\lambda < b\lambda_1$ and $0<a<1/\Lambda$, we deduce from \eqref{3.11} that $\mu_1^2 \geq 4\mu_0 d_3$. Then, we can pick
    \begin{eqnarray}\label{6.4}
    K=\frac{\mu_1 + \sqrt{\mu_1^2-4\mu_0 d_3}}{\mu_0},
   \end{eqnarray}
   and obtain that $\mu_1/\mu_0 \leq K < 2\mu_1/\mu_0$.
  Observing that
  $$J(u_0) = J(kv_0) = - \frac{\mu_0}{4}k^4 + \frac{\mu_1}{2}k^2,
  $$
  and by a direct calculation, we infer that $0<J(u_0) <d_3$, and $u_0 \in \textit{\textbf{W}}_3^{\,-}$ for $K<k^2<2\mu_1/\mu_0$.\par
  Moreover, we denote by $\varepsilon_0 = d_3 - J(u_0)$, thus $\varepsilon_0 >0$ for $k^2 >K$. Fixing $k$ such that $K<k^2<2\mu_1/\mu_0$, and selecting $m< \sqrt{\varepsilon_0}$, then we have
  $$0<E(0)=\frac{m^2}{2}+J(u_0)<d_3,
  $$
  which implies that $(u_0, u_1)$ satisfy $(h3)$.\par
  (iv) \emph {Construct initial data satisfying $(h4)$}. \par
  Assume in addition that $\int_{\Omega} v_0 v_1 \dif x >0$, and let $V_0 = \|v_0\|_2^2$, then $0 \not\in  \textit{\textbf{N}}_3 \cup \textit{\textbf{N}}_3^{\,-}$ indicates that $V_0 >0$. \par
  Fixing $k=k_0$ such that
  $$k_0^2 >K_0:= \max\left\{K, \frac{4d_3}{b_0 \lam_1 V_0}\right\},
  $$
  where $K>0$ is defined by \eqref{6.4}. By virtue of $k_0^2 >K_0$ and $\mu_0 k_0^4 -2\mu_1 k_0^2 +4d_3 >0$, we may choose $m = m_1$ such that
    $$\frac{\mu_0}{2}k_0^4 - \mu_1 k_0^2 +2d_3\, \leq\, m_1^2\, <\, \frac{\mu_0}{2}k_0^4 - \mu_1 k_0^2 + \frac{1}{2}b_0\lam_1 V_0 k_0^2.
    $$
    Selecting $k=k_0$ and $m=m_1$, then we deduce from \eqref{6.3} that
    $$d_3 \leq E(0) < \frac{1}{4}b_0\lam_1 k_0^2 \|v_0\|_2^2 = \frac{1}{4}b_0\lam_1 \|u_0\|_2^2,
    $$
 which implies that $(u_0, u_1)$ satisfy $(h4)$.
 \end{proof}

\emph{Existence of initial data leading to Theorem }2.8:
  \begin{prop}
   Let $p=3$, $b\lam_1 < \lam < b\lam_1 + \delta$, where $\delta>0$ is the number appearing in Theorem $2.8$ {\rm(}see also Lemma $3.8${\rm)}. \par
   {\rm (1)} Assume in addition that $\eqref{2.12}$ satisfies, then there exists such initial data $(u_0, u_1)$ which satisfy $\eqref{1.2}$, $u_0 \in \textit{\textbf{N}}_{3} ^{\, -}$ and $E(0) <d_3^-$, where $d_3^-$ is defined by \eqref{2.13}. \par
   {\rm(2)}  $\eqref{2.12}$ holds provided that $a<1/\tilde{A}$, where $\tilde{A}$ is defined by
    $$\tilde{A} = \inf_{\mbox{\tiny$\begin{array}{c}
   u\in H^1_0(\Omega)\cap H^2 (\Omega)\\
       \|\nabla u\|_2 \neq 0
       \end{array}$}} \frac{~\|\nabla u\|_2^4~}{\|u\|^4_{4}}{\rm .}
    $$
       \end{prop}
 \begin{proof}
(1) Choosing $v_0 \in \left(\textit{\textbf{N}}_3 \cap L^+\right) \cap H^2(\Omega)$, then we have $\|v_0\|_4^4 - a \|\nabla v_0\|_2^4 = b\|\nabla v_0\|_2^2 - \lam \|v_0\|_2^2 >0$. Let $s_0 = \|v_0\|_4^4 - a \|\nabla v_0\|_2^4$ and $u_0 = kv_0$, noting that $d_3^- <0$, then we obtain $J(u_0) = -(k^4 -2k^2)s_0 /4 <d_3^-$ and $I(u_0) = -(k^4 - k^2)s_0 <0$, provided that
 \[k^2 > 1+ \sqrt{1-\frac{4d_3^-}{s_0}}.
 \]
   Let $u_1 \in H_0^1(\Omega)$ with $\|u_1\|_2^2 < d_3^- - J(u_0)$, hence $(u_0, u_1)$ satisfy \eqref{1.2}, $u_0 \in \textit{\textbf{N}}_3^{\,-}$ and $E(0) <d_3^-$.

   (2) Since $a<1/\tilde{A}$ and by the definition of $\tilde{A}$, for any $\varepsilon \in (0,~(1-a\tilde{A})/a)$, we may choose $u_\varepsilon \in H_0^1(\Omega) \cap H^2(\Omega)$ satisfying
 $$a\|\nabla u_\varepsilon\|_2^4 < a(\tilde{A}+\varepsilon) \|u_\varepsilon\|_4^4 < \|u_\varepsilon\|_4^4,
  $$
  which implies that $u_\varepsilon \in S$.
  Then, it follows from Proposition 3.9 (ii) that there exists  $\sigma_{u_\varepsilon} >0$ such that $\sigma_{u_\varepsilon} u_\varepsilon \in \textit{\textbf{N}}_3 \cap L^+$, and \eqref{2.12} holds.

 Then, the proof of Proposition 6.4 is finished.
\end{proof}

\emph{Existence of initial data leading to Theorem }2.9:
\begin{prop}
     If $3<p<5$ and $\lam \leq b\lam_1$, then there exists such initial data $(u_0, u_1)$ which satisfy $\eqref{1.2}$, $u_0 \in \textit{\textbf{W}}_{p} ^{\, +}$ and $E(0)< d_p$.
      \par
 \end{prop}
 \begin{proof}
 Let $u_0 \in H_0^1(\Omega) \cap H^2(\Omega)$, then it follows from Theorem 4.6 that $u_0 \in \textit{\textbf{W}}_{p} ^{\, +} = J^{d_p} \cap \textit{\textbf{N}}_p^{\,+}$ by choosing $0< \|\nabla u_0\| \leq \hat{r}_p$, where $\hat{r}_p$ is defined by \eqref{4.9}. \par
  Furthermore, selecting $u_1 \in H_0^1(\Omega)$ with $\|u_1\|_2^2 < d_p -J(u_0)$ and noting \eqref{1.11}, then we obtain that $(u_0, u_1)$ satisfying \eqref{1.2}, $u_0 \in \textit{\textbf{W}}_{p} ^{\, +}$ and $E(0)< d_p$.
  \end{proof}

\emph{Existence of initial data leading to Theorem }2.10:
\begin{prop}
     Let $p>3$, then there exists such initial data $(u_0, u_1)$ which satisfy $\eqref{1.2}$ and each of the conditions in Theorem $2.10$.
   \end{prop}
\begin{proof}
Choosing $v_0 \in  H_0^1(\Omega) \cap H^2(\Omega)$ with $\|v_0\|_{p+1} >0$, $v_1 \in H_0^1(\Omega)$ with $\|v_1\|_2 =1$, and $(u_0, u_1) = (kv_0, m v_1)$ with $k$, $m \in \mathbb{R}_+$. Then $(u_0, u_1)$ satisfy \eqref{1.2}, and
\begin{eqnarray}\label{6.5}
  E(0) =  \frac{m^2}{2} - Ak^{p+1}  + Bk^4 + Ck^2,
\end{eqnarray}
where
 $$A=\frac{1}{p+1}\|v_0\|_{p+1}^{p+1}, \quad B=\frac{a}{4}\|\nabla v_0\|_2^4,\quad C=\frac{1}{2}\left(b\|\nabla v_0\|_2^2 - \lam \|v_0\|_2^2\right).
 $$

 In what follows, we choose $k$, $m>0$ such that $(u_0, u_1)$ satisfy each of the conditions in Theorem 2.10. To this end, we define $\phi(k) = - Ak^{p+1}  + Bk^4 + Ck^2$.\par
  (i) \emph {Construct initial data satisfying $(a1)$}: $\lam\leq b\lam_1$ and $E(0)<0$. \par
  Noting $A>0$ and $p>3$, we know that there exists $k_1 >0$ such that $\phi(k_1)<-1/2$. Picking $m=1$ and $k=k_1$, we then deduce from \eqref{6.5} that $E(0)<0$ and $(u_0, u_1)$ satisfy $(a1)$.\par
 (ii) \emph {Construct initial data satisfying $(a2)$}: $\lam\leq b\lam_1$, \eqref{2.8} and $E(0)=0$. \par
  We may suppose that $\int_{\Omega} v_0 v_1 \dif x >0$. Choosing $k_2 >0$ such that $\phi(k_2) = -1/2$. Then by \eqref{6.5}, we obtain that $(u_0, u_1)$ satisfy $(a2)$ for $m=1$ and $k=k_2$.\par
  (iii) \emph {Construct initial data satisfying $(a3)$}: $3<p<5$, $\lam\leq b\lam_1$, \eqref{2.8}, $u_0 \in \textit{\textbf{W}}_{p} ^{\, -}$ and $0<E(0)<d_p$. \par

 Taking $(v_0, v_1)$ such that $\int_{\Omega} v_0 v_1 \dif x >0$. By virtue of $p>3$ and $A>0$, we deduce that there exists $\hat{k} >0$ such that
  \begin{eqnarray}\label{6.6}
     J(k v_0) = Bk^4 + C k^2 - A k^{p+1} <d_p,\qquad {\rm for}\ k>\hat{k}.
  \end{eqnarray}
  Furthermore, we can select a number $K>0$ such that $K \|\nabla v_0\|_2 >R_p$, where $R_p$ is defined by \eqref{4.9}. Thanks to Theorem 4.6, by choosing $u_0 = k_3 v_0$ and $k_3 > \max\{\hat{k}, K\}$, we have $J(u_0)<d_p$ and $u_0 \in \textit{\textbf{W}}_{p} ^{\, -}$.

   Moreover, picking $m^2 \in \left(-2J(u_0), 2d_p -2J(u_0)\right)$ and due to \eqref{6.5}, then we arrive at
  $$0<E(0)=\frac{m^2}{2}+J(u_0)<d_p,
  $$
  which implies that $(u_0, u_1)$ satisfy $(a3)$.\par
  (iv) \emph {Construct initial data satisfying $(a4)$}: $\lam \leq b\lam_1$, \eqref{2.8} and \eqref{2.15} hold. \par
  Let $\int_{\Omega} v_0 v_1 \dif x >0$ and pick
  $$D= \frac{(p-3)a\lam_1^2}{4(p+1)} \|v_0\|_2^4.
  $$
   Due to $p>3$ and $A>0$, then there exists $k_4 >0$ such that
  \begin{eqnarray}\label{6.7}
    J(k_4 v_0) = - A k_4^{p+1} + Bk_4^4 + C k_4^2 < D k_4^4.
  \end{eqnarray}
  Choosing $m_1 >0$ such that
 \begin{eqnarray*}
   -2J(k_4 v_0) < m_1^2 < 2D k_4^4 - 2J(k_4 v_0),
 \end{eqnarray*}
  and letting $(u_0, u_1) = (k_4 v_0, m_1 v_1)$, then we deduce from \eqref{6.5} and \eqref{6.7} that $(u_0, u_1)$ satisfy $(a4)$.\par
 (v) \emph {Construct initial data satisfying $(b1)$}: $\lam > b\lam_1$ and $E(0) < -h_0/(2p+2)$.\par
 It follows from $p>3$ and $A>0$ that
 $$J(kv_0) = - Ak^{p+1}  + Bk^4 + Ck^2 \rightarrow -\infty, \qquad {\rm as}\ k\rightarrow \infty,
 $$
 and there exists $k_5 >0$ such that
 $$J(k_5 v_0) < - \frac{1}{2} - \frac{h_0}{2(p+1)}.
 $$
 Let $(u_0, u_1) = (k_5 v_0, v_1)$ and due to \eqref{6.5}, then we know $(u_0, u_1)$ satisfy $(b1)$.\par
 (vi) \emph {Construct initial data satisfying $(b2)$}: $\lam > b\lam_1$, $E(0) = -h_0/(2p+2)$ and \eqref{2.8}.\par
 Using the same way as in proving $(a2)$ and $(b1)$, we can verify the existence of initial data $(u_0, u_1)$ satisfying $(b2)$.
 \end{proof}

\end{document}